\documentclass[11pt, reqno]{amsart}
\usepackage{graphicx,psfrag}
\usepackage[psamsfonts]{amssymb}
\usepackage{pdfsync}
\usepackage{amscd}
\usepackage{amsmath}
\usepackage{amsxtra}
\usepackage{mathrsfs}
\usepackage{stmaryrd}
\usepackage[small,nohug,heads=littlevee]{diagrams}
\usepackage{comment}
\usepackage{enumerate}
\usepackage{hyperref}
\usepackage{xfrac}
\diagramstyle[labelstyle=\scriptstyle]

\usepackage{mathpple}

\usepackage[width=5.8in, height=8.5in, bottom=1.3in, centering]{geometry}

\theoremstyle{plain}
    \newtheorem{theorem}{Theorem}[section]
    \newtheorem{lemma}[theorem]{Lemma}
    \newtheorem{corollary}[theorem]{Corollary}
    \newtheorem{proposition}[theorem]{Proposition}
    \newtheorem*{main}{Main Theorem}
    \newtheorem{assumption}[theorem]{Assumption}
\theoremstyle{definition}
    \newtheorem{definition}[theorem]{Definition}
    \newtheorem{example}[theorem]{Example}

    \newtheorem*{thank}{Acknowledgements}
\theoremstyle{remark}
    \newtheorem{remark}[theorem]{Remark}
    
    \newtheorem{question}[theorem]{Questions}

\numberwithin{equation}{section}

\newcommand{\ZZ}{\mathbb{Z}}

\newcommand{\NN}{\mathbb{N}}

\newcommand{\Econn}{\mathbb{E}}
\newcommand{\Fconn}{\mathbb{F}}
\newcommand{\Mconn}{\mathbb{M}}

\newcommand{\Xconn}{\mathbb{X}}
\newcommand{\Yconn}{\mathbb{Y}}
\newcommand{\Cconn}{\mathbb{C}}

\newcommand{\Edot}{E^{\bullet}}
\newcommand{\Mod}{\operatorname{Mod}}
\newcommand{\CC}{\mathcal{C}}
\newcommand{\real}{\mathbb{R}}
\newcommand{\complex}{\mathbb{C}}
\newcommand{\PP}{\mathcal{P}}
\newcommand{\A}{\mathcal{A}}
\newcommand{\B}{\mathcal{B}}
\newcommand{\C}{\mathcal{C}}
\newcommand{\D}{\mathcal{D}}

\newcommand{\M}{\mathcal{M}}
\newcommand{\AOO}{\A^{0,0}}
\newcommand{\AOK}{\A^{0,k}}
\newcommand{\AOD}{\A^{0,\bullet}}

\newcommand{\Adot}{\mathcal{A}^{\bullet}}
\newcommand{\PA}{\mathcal{P}_{A}}
\newcommand{\qPA}{q\PA}
\newcommand{\gext}{\operatorname{Ext}}

\newcommand{\rank}{\operatorname{rank}}
\newcommand{\Ho}{\operatorname{Ho}}

\newcommand{\Id}{\operatorname{Id}}
\newcommand{\partialbar}{\overline{\partial}}

\newcommand{\Dcat}{\mathcal{D}^b}
\newcommand{\Dcoh}{\mathcal{D}^b_{coh}}

\newcommand{\invlim}{\varprojlim}

\newcommand{\Cinf}{$C^\infty$}
\newcommand{\embed}{i_{Y/X}}
\newcommand{\Spec}{\operatorname{Spec}_\complex}
\newcommand{\aaa}{\mathfrak{a}}

\newcommand{\Shat}{\hat{S}}

\newcommand{\GL}{\mathbf{GL}}

\newcommand{\JetX}{\mathcal{J}^\infty_X}

\newcommand{\Lie}{\mathcal{L}}

\newcommand{\Ob}{\operatorname{Ob}}
\newcommand{\Coh}{\mathbf{Coh}}
\newcommand{\Ker}{\operatorname{Ker}}
\newcommand{\upscript}[1]{{\scriptscriptstyle{#1}}}
\newcommand{\Lnorm}[1]{\lVert#1\rVert_\infty}
\newcommand{\Cone}{\operatorname{Cone}}
\newcommand{\homof}{\operatorname{f}}
\newcommand{\homog}{\operatorname{g}}

\newcommand{\Yhat}{{\hat{Y}}}
\newcommand{\Yhatfinite}{{\hat{Y}^{\upscript{(r)}}}}
\newcommand{\Esheaf}{\mathscr{E}}
\newcommand{\Fsheaf}{\mathscr{F}}
\newcommand{\Osheaf}{\mathscr{O}}
\newcommand{\Gsheaf}{\mathscr{G}}
\newcommand{\Asheaf}{\mathscr{A}}
\newcommand{\Isheaf}{\mathscr{I}}
\newcommand{\Jsheaf}{\mathscr{J}}
\newcommand{\Jideal}{\mathfrak{J}}
\newcommand{\Msheaf}{\mathscr{M}}
\newcommand{\Nsheaf}{\mathscr{N}}
\newcommand{\formal}{^{\upscript{(\infty)}}}

\newcommand{\Cperf}{\CC_{pe}}
\newcommand{\Linf}{L_\infty}

\newcommand{\hprodperf}{\PP_B \times^h_{\PP_D} \PP_C}
\newcommand{\hprodqperf}{q\PP_B \times^h_{q\PP_D} q\PP_C}
\newcommand{\hprodcperf}{\Cperf(Z_1) \times^h_{\Cperf(Z_1 \cap Z_2)} \Cperf(Z_2)}
\newcommand{\dash}{\operatorname{-}}

\newarrow{Equal} =====

\title[The Dolbeault dga of a formal neighborhood]
{The Dolbeault dga of a formal neighborhood}

\author{Shilin Yu}
\address{Department of Mathematics, Pennsylvania State University, University
Park, PA 16802, USA}

\email{yu@math.psu.edu}

\keywords{Formal neighborhoods, Derived categories, Differential graded algebras, Differential graded categories}

\subjclass[2010]{Primary 18D20; Secondary 	14B20, 18E30, 58A20}

\begin{document}

\begin{abstract}
  Inspired by a work of Kapranov \cite{Kapranov}, we define the notion of Dolbeault complex of the formal neighborhood of a closed embedding of complex manifolds. This construction allows us to study coherent sheaves over the formal neighborhood via complex analytic approach, as in the case of usual complex manifolds and their Dolbeault complexes. Moreover, our the Dolbeault complex as a differential graded algebra can be associated with a dg-category according to Block \cite{Block1}. We show this dg-category is a dg-enhancement of the bounded derived category over the formal neighborhood under the assumption that the submanifold is compact. This generalizes a similar result of Block in the case of usual complex manifolds.
\end{abstract}

\maketitle

\tableofcontents

\section{Introduction}

The goal of this paper is to construct the analogue of the Dolbeault complex for the formal neighborhood of a closed complex submanifold inside a complex manifold. The main result below (Theorem \ref{thm:PA_Yhat}) describes the derived category of coherent sheaves on a formal neighborhood and its dg-enhancement using the Dolbeault complex we construct. This is in line with the recent work of Block \cite{Block1}. Our construction gives a new perspective on a fundamental work of Kapranov on the formal neighborhood of a diagonal embedding \cite{Kapranov}.

Recall that the usual Dolbeault complex of a complex manifold $X$ is the data $(\AOD(X), \allowbreak \partialbar)$, where $\AOD(X) = \Gamma(X, \wedge^\bullet \Omega^{0,1}_X)$ is the graded vector space of smooth $(0,q)$-forms over $X$ and $\partialbar$ is the $(0,1)$-component of the de Rham differential. Moreover, $(\AOD(X),\partialbar)$ is naturally a (graded commutative) differential graded algebra (dga) with wedge product as the multiplication, which we will call as the \emph{Dolbeault dga} throughout the paper. The corresponding sheaf of dgas, i.e., the sheaf of smooth $(0,q)$-forms with the differential $\partialbar$, provides a fine resolution of the sheaf of holomorphic functions $\Osheaf_X$. Thus the Dolbeault dga encodes the holomorphic structure of the complex manifold. For example, it is a basic fact that a holomorphic vector bundle can be thought of as a dg-module over the Dolbeault dga with additional projectiveness.

Moreover, coherent analytic sheaves can also be realized as certain modules over the Dolbeault dga. In \cite{Block1}, Block constructs a dg-category $\PA$ for any given dga $A = (\Adot, d)$, which he calls the \emph{perfect category of cohesive modules over $A$}. In the case of the Dolbeault dga $A = (\AOD(X),\partialbar)$ of a compact complex manifold $X$, he has shown that the homotopy category of $\PA$ is equivalent to $\Dcoh(X)$, the bounded derived category of complexes of $\Osheaf_{X}$-modules with coherent cohomology (Theorem \ref{thm:Block_complex_manifold}). In other words, $\PA$ provides a dg-enhancement of $\Dcoh(X)$.

In this paper, we will generalize these results to the case of formal neighborhoods. Recall that, for a closed embedding of complex manifolds $i: X \hookrightarrow Y$, the formal neighborhood $\Yhat$ is the ringed space with $X$ being the underlying topological space and the structure sheaf of rings
\begin{displaymath}
  \Osheaf_{\Yhat} := \varprojlim_{r} \Osheaf_X / \Isheaf^{r+1},
\end{displaymath}
where $\Isheaf$ is the ideal sheaf of holomorphic functions vanishing on $X$. We will define a notion of Dolbeault dga $(\Adot(\Yhat),\partialbar)$ for $\Yhat$ and prove the following analogue of Block's theorem (see Theorem \ref{thm:PA_Yhat}):

\begin{main}
  Suppose $X \hookrightarrow Y$ is a closed embedding of complex manifolds and $X$ is compact. Let $A=(\Adot(\Yhat),\partialbar)$ be the Dolbeault dga of the formal neighborhood $\Yhat$, then the category homotopy category $\Ho \PA$ of the dg-category $\PA$ is equivalent to $\Dcoh(\Yhat)$, the bounded derived category of complexes of sheaves of $\Osheaf_{\Yhat}$-modules with coherent cohomology.
\end{main}

Our definition of Dolbeault dga for formal neighborhoods is inspired by the work of Kapranov \cite{Kapranov}, in which he considers the formal neighborhood of the diagonal embedding $\Delta: X \hookrightarrow X \times X$. The Dolbeault dga in this case is isomorphic the Dolbeault resolution $(\AOD_X (\JetX), \partialbar)$ of the jet bundle $\JetX$ of $X$. Kapranov shows that this dga is the Chevalley-Eilenberg complex of an $L_\infty$-algebra structure, unique up to homotopy equivanlence, on the shifted tangent bundle $TX[-1]$ (or more precisely, on the Dolbeault complex $(\A^{0,\bullet-1}_X(TX), \partialbar)$. In \cite{LinfAlgebroid1}, \cite{LinfAlgebroid2} and subsequent work, we will generalized this result for an arbitrary embedding $i: X \hookrightarrow Y$, in which case we will have an $\Linf$-algebra structure on the shifted normal bundle $N[-1]$. However, here we discover an interesting phenomenon which is absent in the special case of a diagonal embedding: $N[-1]$ is not merely an $\Linf$-algebra but an \emph{$\Linf$-algebroid} with an \emph{$\infty$-anchor map} $N[-1] \to TX$, which is an $\Linf$-morphism from the $\Linf$-algebra $(\A^{0,\bullet-1}_X(N),\partialbar)$ to the dg-Lie algebra $(\AOD_X(TX),\partialbar)$ equipped with the usual Lie bracket of tangent vector fields. 

The $\Linf$-structure on $N[-1]$ has various interesting applications. For instance, it will be shown elsewhere that $N[-1]$ as an $\Linf$-algebra governs the infinitesimal deformations of the embedding $i: X \hookrightarrow Y$. Study on such deformation problem from the viewpoint of Lie theory has been an active topic recently (e.g., \cite{Manetti}, \cite{FiorenzaManetti}, \cite{Ran}), yet no explict formulae for the underlying Lie structures have been presented as far as we know.

Moreover, the $L_\infty$-algebroid $N[-1]$ has close relation to the Yoneda algebra
\begin{equation}\label{eq:Ext}
 \gext^\bullet_Y (\Osheaf_X, \Osheaf_X).
\end{equation}
The key observation is that \eqref{eq:Ext} together with its algebra structure only depends on the infinitesimal behavior of the embedding, i.e., there is a natural isomorphism of algebras
\begin{displaymath}
  \gext^\bullet_Y (\Osheaf_X, \Osheaf_X) \simeq \gext^\bullet_{\Yhat} (\Osheaf_X, \Osheaf_X).
\end{displaymath}
In order to compute the Yoneda algebra \eqref{eq:Ext}, we will need to work in Block's category and find an explicit cohesive modules over $(\Adot(\Yhat),\partialbar)$ which corresponds to $\Osheaf_X$ as a sheaf of $\Osheaf_{\Yhat}$-modules. This requires a fully description of our Dolbeault dga $(\Adot(\Yhat),\partialbar)$ in terms of the geometry of the embedding, which will be studied in \cite{LinfAlgebroid1} and \cite{LinfAlgebroid2}. The Yoneda algebra \eqref{eq:Ext} or its sheafy version can then be regarded as a \emph{universal enveloping algebra of the $\Linf$-algebroid $N[-1]$}, whose precise definition in such generality is yet to be explored in the future. According to this philosophy, it is hence natural and tempting to ask whether some type of Duflo-Kontsevich isomorphism exists in this general setting. 

It should be mentioned that a similar story is also developed in the algebraic setting by Calaque, C\u{a}ld\u{a}raru and Tu in an ongoing work. We appreciate their generous sharing of ideas and helpful remarks on our work.

The paper is organized as follows. In \S~\ref{subsec:definition} we give the definition of the Dolbeault dga $(\Adot(\Yhat),\partialbar)$ of $\Yhat$ and show that the corresponding sheaf $(\Asheaf^\bullet(\Yhat),\partialbar)$ of Dolbeault dgas provides a fine resolution of $\Osheaf_{\Yhat}$. In the rest of \S ~ \ref{sec:Dolbeault_DGA}, we discuss Dolbeault resolutions of vector bundles and, more generally, coherent analytic sheaves over formal neighborhoods. Namely, tensoring any coherent sheaf $\Fsheaf$ on $\Yhat$ with $(\Asheaf^\bullet_{\Yhat},\partialbar)$ produces a fine resolution of $\Fsheaf$ as $\Osheaf_{\Yhat}$-module (Theorem \ref{thm:Dolbeault_main}). However, the proof is technical and not directly related to the rest of the paper, so we postpone it till Appendix \ref{appendix:flatness}. In \S~\ref{sec:cohesive}, we first review the main ingredients of Block's dg-category $\PA$ of cohesive modules \cite{Block1} and descent result of $\PA$ from \cite{Ben-Bassat-Block} (Theorem \ref{thm:descent_perf}). Then we prove the analogue of Block's theorem aforementioned for the formal neighborhood $\Yhat$ in Theorem \ref{thm:Dolbeault_main}, under the assumption that the submanifold $X$ is compact.

\begin{thank}
  We would like to thank Jonathan Block, Damien Calaque, Andrei C\u{a}ld\u{a}raru, Nigel Higson and Junwu Tu for the guidance to this project and many inspiring conversations. This research was partially supported under NSF grant DMS-1101382.
\end{thank}

\section{The Dolbeault dga of a formal neighborhood}\label{sec:Dolbeault_DGA}

\subsection{Definitions and basic properties}\label{subsec:definition}

Let $Y$ be a complex manifold, $X$ a closed complex submanifold of $X$ via the embedding $i : X \hookrightarrow Y$. We denote by $\Isheaf \subset \Osheaf_Y$ the ideal sheaf of germs of holomorphic functions vanishing over $X$. For each $r \geq 0$, we can define \emph{the $r$-th formal neighborhood} $(\Yhatfinite, \Osheaf_{\Yhatfinite})$ with $X$ as the underlying topological space and $\Osheaf_{\Yhatfinite} = \Osheaf_Y / \Isheaf^{r+1}$ the structure sheaf. All these sheaves form an inverse system in an obvious way, and by passing to the inverse limit we get \emph{the (complete) formal neighborhood $\Yhat = \Yhat^{\upscript{(\infty)}}$} with structure sheaf
\begin{displaymath}
 \Osheaf_{\Yhat} = \varprojlim_{r} \Osheaf_{\Yhatfinite} = \varprojlim_{r} \Osheaf_X / \Isheaf^{r+1}.
\end{displaymath}
Sometimes we will also use $X^{\upscript{(\infty)}}_Y$ and $X^{\upscript{(r)}}_Y$ instead of $\Yhat$ and $\Yhatfinite$ to emphasize the submanifolds in question.

Roughly, the sheaf $\Osheaf_X / \Isheaf^{r+1}$ records holomorphic functions on $Y$ defined around $X$ up to $r$-th order. We adopt the same idea for smooth functions and differential forms. Denote by $(\AOD(Y), \partialbar) = (\wedge^\bullet \Omega_Y^{0,1}, \partialbar)$ the Dolbeault complex of $Y$, with the differential
\begin{displaymath}
  \partialbar: \AOD(Y) \to \A^{0,\bullet+1}(Y)
\end{displaymath}
being the $(0,1)$-part of the de Rham differential. In fact, it is a \emph{differential graded algebra (dga)} with multiplication being the wedge product of differential forms (for the official definition of a dga, see Definition \ref{defn:dga}). Moreover, it admits a natural Fr\'echet algebra structure, with respect to which $\partialbar$ is continuous. We call such dga as a \emph{Fr\'echet dga}. For each natural number $r$, we define
\begin{equation}\label{defn:aaa_finite}
  \aaa^k_r = \aaa^k_r(X/Y) := \left\{ \omega \in \A^{0,k}(Y) \,
                                 \left| \,
                                   \begin{array}{c}
                                     i^*(\Lie_{V_1} \Lie_{V_2} \cdots \Lie_{V_l} \omega) = 0, \, \forall~ V_j \in C^\infty(T^{1,0}Y), \\
                                     \forall~ 1 \leq j \leq l, \, 0 \leq l \leq r.
                                   \end{array}
                                 \right.
                               \right\}
\end{equation}
a subset of each $\A^{0,k}(Y)$, $k \in \NN$, where $T^{1,0}Y$ is the $(1,0)$-component of the complexified tangent bundle of $Y$ and $\Lie_V$ denotes the usual Lie derivative of forms.

\begin{proposition}\label{prop:dg-ideal}
  $\aaa^\bullet_r$ is a dg-ideal of $(\AOD(Y),\partialbar)$, i.e., it is invariant under the action of $\partialbar$. Moreover, it is a closed under the Fr\'echet topology.
\end{proposition}

Before proving the proposition, we need some preliminaries on a dg-version of the Cartan calculus in our context. First notice that, if $V$ is a $(1,0)$-vector field and $\omega$ is a $(0,q)$-form over $Y$, then the Lie derivative $\Lie_V \omega$ is still a $(0,q)$-form. Moreover, this operation is $C^\infty(Y)$-linear with respect to $V$, i.e.,
\begin{equation}\label{eq:Lie_linearity}
  \Lie_{gV} \omega = g \cdot \Lie_V \omega, \quad \forall~ g \in C^\infty(Y).
\end{equation}
Indeed, by Cartan's formula,
\begin{displaymath}
  \Lie_V \omega = \iota_V d\omega + d\iota_V\omega.
\end{displaymath}
But $\iota_V\omega = 0$ since we are contracting a $(1,0)$-vector field with a $(0,q)$-form. So we have
\begin{displaymath}
  \Lie_V \omega = \iota_V d\omega = \iota_V (\partial \omega + \partialbar \omega) = \iota_V \partial \omega,
\end{displaymath}
which is $C^\infty(Y)$-linear in $V$. Thus we can extend the Lie derivative into an operator
\begin{displaymath}
  \Lie_{(\cdot)}(\cdot) : \A^{0,k}(T^{1,0}Y) \times \A^{0,l}(Y) \to \A^{0,k+l}(Y), \quad \forall~ k,l \geq 0
\end{displaymath}
by defining
\begin{displaymath}
  \Lie_{\eta \otimes V} \zeta = \iota_{\eta \otimes V} \partial \zeta = \eta \wedge \Lie_V \zeta
\end{displaymath}
for any $\eta,~\zeta  \in \AOD(Y)$ and $V \in C^\infty(T^{1,0}Y)$. The discussions can be packaged into the following lemma:

\begin{lemma}
  There is a homomorphism of dg-Lie algebras
  \begin{displaymath}
    \theta : \A^{0,\bullet}(T^{1,0}Y) \to \mathcal{D}er^\bullet_\complex(\AOD(Y), \AOD(Y))
  \end{displaymath}
  given by
  \begin{displaymath}
    \theta(\eta \otimes V) = \Lie_{\eta \otimes V},
  \end{displaymath}
  which is $\AOD(Y)$-linear. In particular, we have
  \begin{equation}\label{eq:[dbar,L]}
    \Lie_{\partialbar V} \omega = \partialbar \Lie_V \omega - \Lie_V \partialbar \omega,
  \end{equation}
  for any $(1,0)$-vector field $V$ and $(0,q)$-form $\omega$ over $Y$.
\end{lemma}

\begin{proof}[Proof of Proposition \ref{prop:dg-ideal}]
  The first statement of the proposition is immediate from the definition of $\aaa^\bullet_r$, so we only need to prove its invariance under $\partialbar$. Observe that, by the linearity of $\Lie$ \eqref{eq:Lie_linearity}, if we substitute in the definition \eqref{defn:aaa_finite} of $\aaa^\bullet_r$ those Lie derivatives $\Lie_{V_i}$ by $\Lie_{\mathcal{V}_i}$ for any $\mathcal{V}_i \in \AOD_Y(T^{1,0}Y)$, nothing will be changed and we get an equivalent definition of $\aaa^\bullet_r$. By \eqref{eq:[dbar,L]}, the commutator
  \begin{displaymath}
    [\partialbar , \Lie_{V_1} \cdots \Lie_{V_l}] = \sum_{i=1}^{l} \Lie_{V_1} \cdots \Lie_{\partialbar V_i} \cdots \Lie_{V_l}
  \end{displaymath}
  is still a differential operator on $\AOD_Y$ of order $\leq l$. Thus if $\omega \in \aaa^\bullet_r$, then
  \begin{displaymath}
    i^* \Lie_{V_1} \cdots \Lie_{V_l} \partialbar\omega = \partialbar i^* \Lie_{V_1} \cdots \Lie_{V_l} \omega - \sum_{i=1}^{l} i^* \Lie_{V_1} \cdots \Lie_{\partialbar V_i} \cdots \Lie_{V_l} \omega = 0,
  \end{displaymath}
  for any $0 \leq l \leq r$, which means that $\partialbar \omega$ also lies in $\aaa^\bullet_r$. Hence $\aaa^\bullet_r$ is a dg-ideal.
\end{proof}

\begin{definition}
  The \emph{Dolbeault dga of the $r$-th order formal neighborhood of $X$ in $Y$} (or \emph{Dolbeault dga of $\Yhatfinite$}) is the quotient dga
  \begin{displaymath}
    \A^\bullet(\Yhatfinite) := \AOD(Y) / \aaa^\bullet_r.
  \end{displaymath}
  Note that $\A^\bullet(\Yhat^{\upscript{(0)}}) = \AOD(X)$. Moreover, we have a descending filtration of dg-ideals
  \begin{displaymath}
    \aaa^\bullet_0 \supset \aaa^\bullet_1 \supset \aaa^\bullet_2 \supset \cdots
  \end{displaymath}
  which induces an inverse system of dgas
  \begin{displaymath}
    \AOD(X) = \A^\bullet(\Yhat^{\upscript{(0)}}) \leftarrow \A^\bullet(\Yhat^{\upscript{(1)}}) \leftarrow \A^\bullet(\Yhat^{\upscript{(2)}}) \leftarrow \cdots.
  \end{displaymath}
  We can then define the \emph{Dolbeault dga of the complete formal neighborhood of $X$ in $Y$} (or \emph{Dolbeault dga of $\Yhat$}) as the inverse limit
  \begin{displaymath}
    \A^\bullet(\Yhat) = \A^\bullet(\Yhat^{\upscript{(\infty)}}):= \varprojlim_r \A^\bullet(\Yhatfinite).
  \end{displaymath}
  We endow $\A^\bullet(\Yhat)$ with the initial topology which makes it into a Fr\'echet dga.
\end{definition}

\begin{remark}
  Since we are in the $C^\infty$-situation, the natural map
  \begin{displaymath}
    \AOD(Y) \Big/ \bigcap_{r \in \NN} \aaa^\bullet_r \to \A^\bullet(\Yhat)
  \end{displaymath}
  is in fact an isomorphism of dgas. This is a corollary of a result by E. Borel \cite{Borel}, which essentially states that any formal power series (not necessarily convergent) is the Taylor expansion of some $C^\infty$-function.
\end{remark}

Suppose there are two closed embeddings of complex manifolds $X \hookrightarrow Y$ and $X' \hookrightarrow Y'$, together with a holomorphic map $f: Y \to Y'$ which maps $X$ into $X'$. Then obviously the pullback morphism $f^*: \AOD(Y') \to \AOD(Y)$ maps the dg-ideal $\aaa^\bullet_r(X'/Y')$ into $\aaa^\bullet_r(X/Y)$, hence induces homomorphisms of dgas
\begin{displaymath}
  f^*: \A^\bullet(X_{Y'}^{\prime\upscript{(r)}}) \to \A^\bullet(X_{Y}^{\upscript{(r)}})
\end{displaymath}
for $r \in \NN$ or $r = \infty$. Moreover, if $X'' \hookrightarrow Y''$ is a third embedding and $g: Y' \to Y''$ maps $X'$ into $X''$, then
\begin{displaymath}
  f^* \circ g^* = (g \circ f)^*: \A^\bullet(X_{Y''}^{\prime\prime\upscript{(r)}}) \to \A^\bullet(X_{Y}^{\upscript{(r)}}).
\end{displaymath}
In other words, for each $r$ the association of Dolbeault dgas gives a functor from the category of closed embeddings of complex manifolds to the category of (Fr\'echet) dgas.

There are sheafy versions of the constructions above. Denote by $\Asheaf^{0,\bullet}_Y$ the sheaf of dgas of $(0,q)$-forms on $Y$ with the Dolbeault differential $\partialbar$, i.e., for any open subset $U \subset Y$,
\begin{displaymath}
  \Gamma(U, \Asheaf^{0,\bullet}_Y) = \AOD(U).
\end{displaymath}
The complex of sheaves $(\Asheaf^{0,\bullet}(Y), \partialbar)$ is a fine resolution of $\Osheaf_Y$. Inside $\Asheaf^{0,\bullet}(Y)$ there is a subsheaf $\widetilde{\aaa}^\bullet_r = \widetilde{\aaa}^\bullet_{X/Y,r}$ of dg-ideals of $\Asheaf^{0,\bullet}(Y)$, whose sections over any open subset $U \subset Y$ are
\begin{displaymath}
  \Gamma(U,\widetilde{\aaa}^\bullet_r) = \aaa^\bullet_r(X \cap U / U).
\end{displaymath}
It is a also a fine sheaf and its global sections over $Y$ are just $\aaa^\bullet_r$.

\begin{definition}
  \emph{The sheaf of Dolbeault dgas of $\Yhatfinite$} is defined as the quotient sheaf of dgas
\begin{displaymath}
  \Asheaf^\bullet_{\Yhatfinite} = \Asheaf^{0,\bullet}_Y / \widetilde{\aaa}^\bullet_r
\end{displaymath}
with the differential $\partialbar$ inherited from that of $\Asheaf^{0,\bullet}(Y)$.
\end{definition}

Since $(\widetilde{\aaa}^\bullet_r)_x = \Asheaf^{0,\bullet}(Y)_x$ for $x \not\in X$, $\Asheaf^\bullet_{\Yhatfinite}$ is supported on $X$, so we can think of $\Asheaf^\bullet_{\Yhatfinite}$ as a sheaf over $X$. Then for any open subset $V \subset X$, we have
\begin{displaymath}
  \Gamma(V, \Asheaf^\bullet_{\Yhatfinite}) = \Gamma(U, \Asheaf^{0,\bullet}_Y) / \Gamma(U, \widetilde{\aaa}^\bullet_{X/Y, r}) = \A^{0,\bullet}(U) / \aaa ^\bullet_r (V / U) = \A^\bullet(V^{(r)}_U),
\end{displaymath}
where $U$ is arbitrary open subset of $Y$ such that $U \cap X = V$ and $\A^\bullet(V^{(r)}_U)$ does not depend on the choice of $U$. We have an inverse system of sheaves of dgas
\begin{displaymath}
  \Asheaf^{0,\bullet}_X = \Asheaf^\bullet_{\Yhat^{\upscript{(0)}}} \leftarrow \Asheaf^\bullet_{\Yhat^{\upscript{(1)}}} \leftarrow \Asheaf^\bullet_{\Yhat^{\upscript{(2)}}} \leftarrow \cdots.
\end{displaymath}

\begin{definition}
  \emph{The sheaf of Dolbeault dgas of $\Yhat$} is defined as
\begin{displaymath}
  \Asheaf^\bullet_{\Yhat} = \Asheaf^\bullet_{\Yhat^{\upscript{(\infty)}}} := \varprojlim_{r} \Asheaf^\bullet_{\Yhatfinite}.
\end{displaymath}
\end{definition}

In particular, global sections of $\Asheaf^\bullet_{\Yhatfinite}$ over $X$ are just the Dolbeault dga $\A^\bullet(\Yhatfinite)$ we have defined, for $r \in \NN$ or $r = \infty$.

For $r \in \NN$, there are natural inclusions of sheaves $\Isheaf^{r+1} \hookrightarrow \widetilde{\aaa}^0_r$ which gives a cochain complex of sheaves
\begin{displaymath}
  0 \to \Isheaf^{r+1} \to \widetilde{\aaa}^0_r \xrightarrow{\partialbar} \widetilde{\aaa}^1_r \xrightarrow{\partialbar} \cdots \xrightarrow{\partialbar} \widetilde{\aaa}^{n}_r \to 0,
\end{displaymath}
which is a subcomplex of
\begin{displaymath}
  0 \to \Osheaf_Y \to \Asheaf^{0,0}_Y \xrightarrow{\partialbar} \Asheaf^{0,1}_Y \xrightarrow{\partialbar} \cdots \xrightarrow{\partialbar} \Asheaf^{0,n}_Y \to 0
\end{displaymath}
where $n = \dim_\complex Y$. The complex of quotients is
\begin{displaymath}\label{exsq:resolution_sheaves}
  0 \to \Osheaf_{\Yhatfinite} \to \Asheaf^0_{\Yhatfinite} \xrightarrow{\partialbar} \Asheaf^1_{\Yhatfinite} \xrightarrow{\partialbar} \cdots \xrightarrow{\partialbar} \Asheaf^n_{\Yhatfinite} \to 0.
\end{displaymath}
We will see in the following example and Proposition \ref{prop:flatness_Yhatfinite} that the complex \eqref{exsq:resolution_sheaves} actually terminates at $\Asheaf^m_{\widehat{Y}^{(r)}}$ where $m$ is the dimension of $X$, and it is also exact.

\begin{example}\label{ex:local_dolbeault}
  Let $Y = \complex^{n+m} = \complex^n \times \complex^m$ with linear coordinates $(z_1, \ldots, z_{n+m})$ and $X = 0 \times \complex^m$ be the linear subspace of $Y$ defined by equations $z_1 = \cdots = z_n = 0$. As usual $i : X \hookrightarrow Y$ denotes the embedding. Define the algebra of polynomials up to order $r$
  \begin{displaymath}
    \mathcal{F}^{(r)}_n = \complex [z_1, \ldots, z_n] / (z_1, \ldots, z_n)^{r+1}.
  \end{displaymath}
  We have a homomorphism of dgas
  \begin{equation}\label{eq:hom_T_r}
    \begin{aligned}
    \mathcal{T}_r :~ (\AOD(Y),\partialbar) ~ &\to ~ (\AOD(X) \otimes_{\complex} \mathcal{F}^{(r)}_n, \partialbar \otimes 1),   \\
       \omega  ~ &\mapsto    ~ \sum_{|\alpha| \leq r} \frac{1}{\alpha!} i^*(\partial^\alpha \omega) \otimes z^\alpha
    \end{aligned}
  \end{equation}
  where $\alpha=(\alpha_1, \ldots, \alpha_n)$ are multi-indices of nonnegative integers, $\partial^\alpha = \partial_1^{\alpha_1} \partial_2^{\alpha_2} \cdots \partial_n^{\alpha_n}$ with $\partial_j = \Lie_{\partial / \partial z_j}$ and $z^\alpha = z_1^{\alpha_1} z_2^{\alpha_2} \cdots z_n^{\alpha_n}$. Obviously $\mathcal{T}_r$ is surjective. Moreover, the kernel of $\mathcal{T}$ is exactly $\aaa^\bullet_r$. To see this, we observe that in the definition \eqref{defn:aaa_finite} of $\aaa^\bullet_r$, we only need to take those test vector fields $V_i$ in the transversal direction to $X = 0 \times \complex^m$ and they can even be the constant vector fields $\partial / \partial z_i$. Thus we
  have an isomorphism of dgas
  \begin{equation}\label{eq:isom_T_r}
    \widetilde{\mathcal{T}}_r: \A^\bullet(\Yhatfinite) \xrightarrow{\simeq} \AOD(X) \otimes_{\complex} \mathcal{F}^{(r)}_n
  \end{equation}
  Take inverse limits of the domains and codomains of $\widetilde{\mathcal{T}_r}$ and note that $\widetilde{\mathcal{T}_r}$ are compatible with the two inverse systems, we hence obtain an isomorphism between Fr\'echet dgas
  \begin{displaymath}
    \widetilde{\mathcal{T}}_{\infty}: ~ \A^\bullet(\Yhat) \xrightarrow{\simeq} \AOD(X) \llbracket z_1, \ldots, z_n \rrbracket.
  \end{displaymath}
  Here $\AOD(X) \llbracket z_1, \ldots, z_n \rrbracket$ is the algebra of formal power series in variables $z_1, \ldots, z_n$ with coefficients in $\AOD(X)$ and is endowed with the initial Fr\'echet algebra structure induced by the inverse limit. In particular, we see that the Dolbeault dgas actually have at most $m = \dim_\complex X$ nonzero components. In other words, $\aaa^k_r = \A^{0,k}(Y)$ for $k > m$.

  To be more explicit, we write
  \begin{displaymath}
    \omega = \sum_{I \subset \{1,\ldots,n+m\}} \omega_I d\overline{z}_I
      =  \sum_{I \subset \{1,\ldots,n\}} \omega_I d\overline{z}_I
       + \sum_{I \not \subset \{1,\ldots,n\}} \omega_I d\overline{z}_I.
  \end{displaymath}
  where $\omega_I \in C^\infty(Y)$. Note that the second term on the right hand side of the equality already lies in $\aaa^\bullet_r$. Then
  \begin{displaymath}
    \mathcal{T}_r(\omega)
      = \sum_{I \subset \{1,\ldots,n\}} \sum_{|\alpha| \leq r}
          \frac{1}{\alpha!} (\partial^\alpha \omega_I)|_X \cdot d\overline{z}_I \otimes z^\alpha
  \end{displaymath}
  where $|\alpha| = |(\alpha_1,\ldots,\alpha_n)| = \sum_{i=1}^{n} \alpha_i$ and $\alpha ! = \alpha_1 ! \alpha_2 ! \cdots \alpha_n !$. A smooth function $f \in \A^{0,0}(Y)$ belongs to $\aaa^0_r$ if and only if it is of the form
  \begin{displaymath}
    f = \sum_{|\alpha|=r+1} z^\alpha g_\alpha + \sum_{i=1}^{n} \overline{z}_i h_i
  \end{displaymath}
  for some $g_\alpha, h_i \in C^\infty(Y)$. In general, $\omega \in \aaa^\bullet_r$ if and only if
  \begin{displaymath}
    \omega \equiv \sum_{I \subset \{1,\ldots,n\}} \omega_I d\overline{z}_I \quad (\text{mod} ~ d\overline{z}_{n+1}, \ldots, d\overline{z}_m)
  \end{displaymath}
  where $\omega_I \in \aaa^0_r$. In other words, for $k \geq 1$,
  \begin{displaymath}
    \aaa^k_r = \sum_{I \subset \{1,\ldots,n\}} \aaa^0_r \cdot d\overline{z}_I + \sum_{i=n+1}^{m} d\overline{z}_i \wedge \A^{0,k-1}(Y).
  \end{displaymath}

  Note that all the arguments above work if we substitute $Y$ by any open subset $U \subset \complex^{n+m}$ and $X$ by the intersection of $U$ with $0 \times \complex^m$. So there are also isomorphisms on the level of sheaves:
  \begin{displaymath}
    \Asheaf^\bullet_{\Yhatfinite} \simeq \Asheaf^{0,\bullet}_X \otimes_{\complex} \mathcal{F}^{(r)}_n
  \end{displaymath}
  and
  \begin{displaymath}
    \Asheaf^\bullet_{\Yhat} \simeq \Asheaf^{0,\bullet}_X \llbracket z_1, \ldots, z_n \rrbracket.
  \end{displaymath}
  where $Y=U \subset \complex^{n+m}$ and $X = U \cap (0 \times \complex^m)$. Moreover, there is also an isomorphism
  \begin{displaymath}
    \Osheaf_{\Yhatfinite} \xrightarrow{\simeq} \Osheaf_X \otimes_{\complex} \mathcal{F}^{(r)}_n
  \end{displaymath}
  defined in exactly the same way as that of $\widetilde{\mathcal{T}}_r$, which fits into an isomorphism between two complexes of sheaves
  \begin{equation}\label{exsq:resolution_sheaves}
    0 \to \Osheaf_{\Yhatfinite} \to \Asheaf^0_{\Yhatfinite} \xrightarrow{\partialbar} \Asheaf^1_{\Yhatfinite} \xrightarrow{\partialbar} \cdots \xrightarrow{\partialbar} \Asheaf^m_{\Yhatfinite} \to 0
  \end{equation}
  and
  \begin{displaymath}
    0 \to \Osheaf_X \otimes  \mathcal{F}^{(r)}_n \to \Asheaf^{0,0}_X \otimes \mathcal{F}^{(r)}_n \xrightarrow{\partialbar \otimes 1} \Asheaf^{0,1}_X \otimes \mathcal{F}^{(r)}_n \xrightarrow{\partialbar \otimes 1} \cdots \xrightarrow{\partialbar \otimes 1} \Asheaf^{0,m}_X \otimes \mathcal{F}^{(r)}_n \to 0.
  \end{displaymath}
 The latter is exact by the Dolbeault-Grothendieck Lemma, thus we also have exactness of the first complex. Since exactness is a local property, we have the following proposition.
\end{example}

\begin{proposition}\label{prop:resolution}
  For any closed embedding $i: X \hookrightarrow Y$ of complex manifolds, the complex of sheaves $(\Asheaf^\bullet_{\Yhatfinite},\partialbar)$ is a fine resolution of $\Osheaf_{\Yhatfinite}$ over $X$, where $r \in \NN$ or $r = \infty$.
\end{proposition}

\begin{proof}
  The case when $r$ is finite was already taken care of by Example \ref{ex:local_dolbeault}. For $r = \infty$, it suffices to show that,  for any local chart $U$ of $Y$ and $V = U \cap X$ such that the pair $(V, U)$ is biholomorphic to a pair of polydiscs $(D' , D)$ where $D \subset \complex^{d+m}$ and $D' = D \cap 0 \times \complex^m$, the complex of sections of \eqref{exsq:resolution_sheaves} over $V$
  \begin{displaymath}
    0 \to \Gamma(V, \Osheaf_{\Yhat}) \to \Gamma(V, \Asheaf^0_{\Yhat}) \xrightarrow{\partialbar} \Gamma(V, \Asheaf^1_{\Yhat}) \xrightarrow{\partialbar} \cdots \xrightarrow{\partialbar} \Gamma(V, \Asheaf^m_{\Yhat}) \to 0
  \end{displaymath}
  is exact, or equivalently, the sequence
  \begin{equation}\label{exsq:resolution_disc_infinite}
    0 \to \Gamma(D', \Osheaf_{{D'}\formal_D}) \to \A^0({D'}\formal_D) \xrightarrow{\partialbar} \A^1({D'}\formal_D) \xrightarrow{\partialbar} \cdots \xrightarrow{\partialbar} \A^m({D'}\formal_D) \to 0
  \end{equation}
  is exact. Note that for any nonnegative integer $s$, the sequence
  \begin{equation}\label{exsq:resolution_disc_finite}
    0 \to \Gamma(D', \Osheaf_{D'^{(s)}_D}) \to \A^0(D'^{(s)}_D) \xrightarrow{\partialbar} \A^1(D'^{(s)}_D) \xrightarrow{\partialbar} \cdots \xrightarrow{\partialbar} \A^m(D'^{(s)}_D) \to 0
  \end{equation}
  is isomorphic to
  \begin{displaymath}
    0 \to \Gamma(D', \Osheaf_{D'}) \otimes  \mathcal{F}^{(s)}_n \to \A^{0,0}(D')\otimes \mathcal{F}^{(s)}_n \xrightarrow{\partialbar \otimes 1} \cdots \xrightarrow{\partialbar \otimes 1} \A^{0,m}(D') \otimes \mathcal{F}^{(s)}_n \to 0
  \end{displaymath}
  by Example \ref{ex:local_dolbeault}. The latter is exact by a stronger version of the Dolbeault-Grothendieck Lemma (see, e.g., p. 25, \cite{GriffithHarris}) or Cartan's Theorem B, hence so is \eqref{exsq:resolution_disc_finite}. Let $s$ vary and we obtain an inverse system of exact sequences. Since all the connecting homomorphisms are surjective, the inverse system satisfies the Mittag-Leffler condition and hence its inverse limit \eqref{exsq:resolution_disc_infinite} is exact.
\end{proof}

\begin{corollary}
  For any nonnegative integer $r$, the sheaf of dg-ideals $(\widetilde{\aaa}^\bullet_r,\partialbar)$ over $Y$ is a fine resolution of $\Isheaf^{r+1}$.
\end{corollary}

From now on, we will often write $\A(X) = \AOO(X)$ and $\A(\Yhatfinite) = \A^0(\Yhatfinite)$ for abbreviation. The same for the corresponding sheaves of dgas.

\subsection{Dolbeault resolutions of coherent sheaves}\label{subsec:Dolbeault_resolution}

In this section, we develop analogues of Dolbeault resolutions for coherent analytic sheaves over formal neighborhoods of all orders. Our argument relies on the following fundamental result due to Malgrange \cite{Malgrange}.

\begin{theorem}\label{thm:Malgrange}
  The ring of germs of complex differentiable functions at a point of a complex manifold is faithfully flat over the ring of germs of holomorphic functions at the same point. In our notation, this means that the stalk $\Asheaf^{0,0}_{X,x}$ is faithfully flat over $\Osheaf_{X,x}$ for any point $x$ of a complex manifold $X$.
\end{theorem}

We first state and prove the analogues for formal neighborhoods of finite orders, whose proofs are straightforward.

\begin{proposition}\label{prop:flatness_Yhatfinite}
  For any $\Yhatfinite$ with finite $r$ and any point $x \in X$, the stalks $\Asheaf^k_{\Yhatfinite,x}$ are faithfully flat $\Osheaf_{\Yhatfinite,x}$-modules for all $k \geq 0$.
\end{proposition}

\begin{proof}
  It is immediate from Example \ref{ex:local_dolbeault} that we have an isomorphism of dgas
  \begin{displaymath}
    \Asheaf^\bullet_{\Yhatfinite,x} \simeq \Asheaf^{0,\bullet}_{X,x} \otimes_\complex \mathcal{F}^{(r)}_n
  \end{displaymath}
  together with an isomorphism
  \begin{displaymath}
    \Osheaf_{\Yhatfinite,x} \simeq \Osheaf_{X,x} \otimes_\complex \mathcal{F}^{(r)}_n
  \end{displaymath}
  which are compatible with the canonical inclusion $\Osheaf_{\Yhatfinite,x} \hookrightarrow \Asheaf_{\Yhatfinite,x}$. Here $\mathcal{F}^{(r)}_n = \complex [z_1, \ldots, z_n] / (z_1, \ldots, z_n)^{r+1}$ where $n$ is the codimension of $X$. The faithful flatness of $\Asheaf_{X,x}$ over $\Osheaf_{X,x}$ immediately implies that of $\Asheaf_{\Yhatfinite,x}$ over $\Osheaf_{\Yhatfinite,x}$, since
  \begin{displaymath}
    \Asheaf_{\Yhatfinite,x} \simeq \Asheaf_{X,x} \otimes_\complex \mathcal{F}^{(r)}_n = \Asheaf_{X,x} \otimes_{\Osheaf_{X,x}} (\Osheaf_{X,x} \otimes_\complex \mathcal{F}^{(r)}_n) \simeq \Asheaf_{X,x} \otimes_{\Osheaf_{X,x}} \Osheaf_{\Yhatfinite,x}.
  \end{displaymath}
  The case of $\Asheaf^k_{\Yhatfinite,x}$ for $k>0$ follows from the fact that $\Asheaf^k_{\Yhatfinite,x}$ is free over $\Asheaf^0_{\Yhatfinite,x}$.
\end{proof}

\begin{proposition}\label{prop:Dolbeault_resolution_Yhatfinite}
  Suppose $\Fsheaf$ is any sheaf of $\Osheaf_\Yhatfinite$-modules, then tensoring it with $(\Asheaf^\bullet_\Yhatfinite, \partialbar)$ gives an exact sequence of $\Osheaf_\Yhatfinite$-modules,
  \begin{displaymath}
    0 \to \Fsheaf \to \Fsheaf \otimes_{\Osheaf_\Yhatfinite} \Asheaf^0_\Yhatfinite \to \Fsheaf \otimes_{\Osheaf_\Yhatfinite} \Asheaf^1_\Yhatfinite \to \cdots \to \Fsheaf \otimes_{\Osheaf_\Yhatfinite} \Asheaf^m_\Yhatfinite \to 0,
  \end{displaymath}
  where $m = \dim X$.
\end{proposition}

\begin{proof}
  This is immediate by the exactness of the complex of stalks $\Osheaf_{\Yhatfinite,x} \hookrightarrow (\Asheaf^\bullet_{\Yhatfinite,x}, \partialbar)$ (Prop. \ref{prop:resolution}) and the flatness of $\Asheaf^\bullet_{\Yhatfinite,x}$ over $\Osheaf_{\Yhatfinite,x}$.
\end{proof}

We now state analogues of the results above for the complete formal neighborhood. The crucial difference here is that we can only obtain exactness of the functor $~ \dash \otimes_{\Osheaf_{\Yhat}} \Asheaf^\bullet_{\Yhat}$ for coherent $\Osheaf_{\Yhat}$-modules, which nevertheless is enough for applications. Also the proofs are much more involved than the case of finite order. To prevent the reader from distraction, we put the proofs in Appendix \ref{appendix:flatness} (see Theorem \ref{thm:exactness_tensor_Dolbeault} and Theorem \ref{thm:Dolbeault_resolution_Yhat}).

\begin{theorem}\label{thm:Dolbeault_main}
  Let $\Fsheaf$ be a coherent $\Osheaf_{\Yhat}$-module. Then $\Fsheaf$ is quasi-isomorphic to the complex $\Fsheaf \otimes_{\Osheaf_{\Yhat}} \Asheaf^\bullet_{\Yhat}$ with the differential $1 \otimes \partialbar$. Moreover, $\Fsheaf \mapsto \Fsheaf \otimes_{\Osheaf_{\Yhat}} \Asheaf^\bullet_{\Yhat}$ gives an exact functor from $\Coh(\Yhat)$ to the category of sheaves of dg-modules over $(\Asheaf^\bullet_{\Yhat}, \partialbar)$.
\end{theorem}

In other words, every coherent analytic sheaf $\Fsheaf$ over $\Yhat$ can be resolved by a soft sheaf $\Fsheaf \otimes_{\Osheaf_{\Yhat}} \Asheaf^\bullet_{\Yhat}$ of dg-$\Asheaf^\bullet_{\Yhat}$-modules. On the other hand, any soft sheaf is characterized by its global sections. Therefore we can characterize $\Fsheaf$ by the dg-$\A^\bullet(\Yhat)$-module $(\Gamma(X, \Fsheaf \otimes_{\Osheaf_{\Yhat}} \Asheaf^\bullet_{\Yhat}), 1 \otimes \partialbar)$, which together with its sheaf version we call the \emph{Dolbeault resolution of $\Fsheaf \in \Coh(\Yhat)$}.

\subsection{Holomorphic vector bundles over a formal neighborhood}\label{subsec:hol_vector_bundle}

In this section we study holomorphic vector bundles over the formal neighborhood $\Yhat$. It was shown by Koszul and Malgrange \cite{KoszulMalgrange} that a holomorphic vector bundle $\Esheaf$ over a usual complex manifold $X$ is the same thing as a \Cinf-vector bundle $\Esheaf_\infty$ with a flat $\partialbar$-connection, i.e., an operator
\begin{displaymath}
  \partialbar_{\Esheaf} : \Gamma(X,\Esheaf_\infty) \to \A^{0,1}(X) \otimes_{\A(X)} \Gamma(X,\Esheaf_\infty)
\end{displaymath}
such that $\partialbar_\Esheaf (f s) = \partialbar(f) s + f \partialbar_{\Esheaf} (s)$ for $f \in \A(X), s \in \Gamma(X,\Esheaf_\infty)$ and satisfying the flatness condition $\partialbar_{\Esheaf} \circ \partialbar_{\Esheaf} = 0$. Here $\Gamma(X,\Esheaf_\infty)$ denotes the global \Cinf-sections of the vector bundle. If $\Esheaf$ is understood as a locally free sheaf of $\Osheaf_X$-modules of finite type, then we have $\Esheaf_\infty = \Asheaf_X \otimes_{\Osheaf_X} \Esheaf$ and $\partialbar_{\Esheaf} = \partialbar \otimes 1$. The flatness condition of $\partialbar_{\Esheaf}$ can be shown to be equivalent to the \emph{integrability} of $\partialbar_{\Esheaf}$, i.e., for any point $x \in X$ there exists a neighborhood $U$ of $x$ over which there is a trivialization $\tau: \Esheaf_\infty |_U \xrightarrow{\simeq} (\Asheaf_{X})^d$, $d = \rank \Esheaf_\infty$, such that $\tau ~ \partialbar_{\Esheaf} \tau^{-1}$ coincides with the product $\partialbar$-connection on vector-valued functions. See, for instance, Theorem 2.1.53, Section 2.1.5, p. 45, \cite{Donaldson}.

We now generalize this result to the case of the formal neighborhood $\Yhat$. By a holomorphic vector bundle $\Esheaf$ over $\Yhat$, we mean a locally free $\Osheaf_{\Yhat}$-modules of finite type. Set $\Esheaf_\infty = \Asheaf_{\Yhat} \otimes_{\Osheaf_{\Yhat}} \Esheaf$, which is regarded as the sheaf of ``\Cinf-sections" of $\Esheaf$ or the underlying ``\Cinf-vector bundle" of $\Esheaf$, while $\Asheaf_{\Yhat}$ is considered as the sheaf of ``smooth functions" over $\Yhat$. By Theorem \ref{thm:Dolbeault_resolution_Yhat}, we have a flat $\partialbar$-connection
\begin{displaymath}
  \partialbar_{\Esheaf} : \Esheaf_\infty \to \Asheaf^1_{\Yhat} \otimes_{\Asheaf_{\Yhat}} \Esheaf_\infty
\end{displaymath}
satisfying $\partialbar_{\Esheaf} \circ \partialbar_{\Esheaf} = 0$, such that its kernel is isomorphic to $\Esheaf$ as an $\Osheaf_{\Yhat}$-module. $\Esheaf_\infty$ is a locally free $\Asheaf_{\Yhat}$-module of finite type.

We want to prove the inverse, that is, for any flat connection $\partialbar_{\Esheaf} : \Esheaf_\infty \to \Asheaf^1_{\Yhat} \otimes_{\Asheaf_{\Yhat}} \Esheaf_\infty$ of a locally free $\Asheaf_{\Yhat}$-module $\Esheaf_\infty$ of finite type, its flat sections form a subsheaf of $\Esheaf_\infty$ which is locally free over $\Osheaf_{\Yhat}$ and of finite type, therefore gives a holomorphic vector bundle. It is enough to show that $\partialbar_{\Esheaf}$ is \emph{integrable} as defined similarly in the usual case, i.e., for any point $x \in X$ there exists a neighborhood $U$ of $x$ over which there is a trivialization $\tau: \Esheaf_\infty |_U \xrightarrow{\simeq} (\Asheaf_{\Yhat})^d$ such that $\tau ~ \partialbar_{\Esheaf} \tau^{-1} = \partialbar$.

\begin{theorem}\label{thm:vectorbundle-integrable}
  A $\partialbar$-connection $\partialbar_\Esheaf$ on a locally free $\Asheaf_{\Yhat}$-module $\Esheaf_\infty$ of finite rank is integrable if and only if $\partialbar_\Esheaf$ is flat, i.e., $\partialbar_\Esheaf \circ \partialbar_\Esheaf = 0$.
\end{theorem}

\begin{proof}
  The definition of integrability automatically implies flatness of the connection. Thus we only need to show that every flat connection is integrable. We follow the argument in Section 2.2.2, p. 50, \cite{Donaldson} with appropriate adjustments. Since the problem is local, by Example \ref{ex:local_dolbeault} we can assume that the submanifold $X$ is a polydisc $D=\{ |w_i| < 1 \} \subset \complex^m$ over which
  \begin{displaymath}
    \Osheaf_{\Yhat} = \Osheaf_X \llbracket z_1, \ldots, z_n \rrbracket, \quad \Asheaf_{\Yhat} = \Asheaf_X \llbracket z_1, \ldots, z_n \rrbracket,
  \end{displaymath}
  and a trivialization of $\Esheaf_\infty$ is fixed, so that the connection $\partialbar_\Esheaf$ is in the form $\partialbar + \alpha$ for a matirx $\alpha$ of formal $(0,1)$-forms valued in formal power series over $D$, i.e.,
  \begin{displaymath}
    \alpha \in M(d, \A^1(D) \llbracket z_1, \ldots, z_n \rrbracket)
  \end{displaymath}
  where $d = \rank \Esheaf_\infty$. The flatness condition now reads as
  \begin{equation}\label{eq:flatness_hol_vb}
    \partialbar \alpha + \alpha \wedge \alpha = 0.
  \end{equation}
  We want to show that there is a smaller polydisc $D_r = \{ |w_i| < r \}$ and a gauge transformation
  \begin{displaymath}
    g \in \GL(d, \A(D_r) \llbracket z_1, \ldots, z_n \rrbracket) \quad \text{or} \quad g : D_r \to \GL(d, \complex \llbracket z_1, \ldots, z_n \rrbracket)
  \end{displaymath}
  such that
  \begin{equation}
    g (\partialbar + \alpha) g^{-1} = \partialbar \quad \text{or} \quad \partialbar g = g \alpha.
  \end{equation}
  We first begin with the special case when the complex dimension of $X$ is one. Then the flatness condition is automatically satisfied. We have a single coordinate on $X$ which we denote as $w$, and we write $\alpha = \rho d \overline{w}$. So $\rho$ is a smooth function over $D$ valued in $d \times d$ matrices with entries as formal power series in variables $z_1, \ldots, z_n$. We want to solve the equation
  \begin{equation}\label{eq:gauge}
    \frac{\partial g}{\partial \overline{w}} - g \rho = 0
  \end{equation}
  with $g$ invertible, and any solution in an arbitrarily small neighborhood of the origin will fulfill our goal. We have decompositions
  \begin{displaymath}
    \rho (w) = \sum_{|I| \geq 0} \rho_I(w) z_I = \rho_0 (w) + \sum_{i=1}^{n} \rho_i(w) z_i + \sum_{1 \leq i \leq j \leq n} \rho_{ij}(w) z_i z_j + \cdots
  \end{displaymath}
  and
  \begin{displaymath}
    g (w) = \sum_{|J| \geq 0} g_J(w) z_J = g_0 (w) + \sum_{i=1}^{n} g_i(w) z_i + \sum_{1 \leq i \leq j \leq n} g_{ij}(w) z_i z_j + \cdots,
  \end{displaymath}
  where $I$, $J$ are multi-indices and the coefficients $\rho_I, g_I$ are smooth functions over $D$ with values in $M(d, \complex)$. Note that $g$ is invertible if and only if $g_0$ is invertible as a matrix-valued function. Therefore solving equation \eqref{eq:gauge} amounts to solving countably many differential equations
  \begin{equation}\label{eq:gauge_family}
    \frac{\partial g_I}{\partial \overline{w}} = g_I \rho_0 + \sum_{\substack{I = J + K  \\ |K| > 0}} g_J \rho_K.
  \end{equation}
  We will solve these equations by induction on the sum $|I|$ of the multi-indices $I$. We first deal with the case when $|I| = 0$, i.e., the equation
  \begin{equation}\label{eq:gauge_rho_0}
    \frac{\partial g_0}{\partial \overline{w}} = g_0 \rho_0
  \end{equation}
  where $g_0$ is required to be invertible. In fact, finding such solutions is equivalent to proving the integrability theorem for the pullback vector bundle $\hat{i}^* \Esheaf_\infty$ where $\hat{i}: X \hookrightarrow \Yhat$ is the embedding of $X$ into the formal neighborhood $\Yhat$. This situation is taken care of by the classical version of the theorem. We sketch the proof from \S 2.2.2, \cite{Donaldson}. By rescaling and multiplying by a cut-off function, we can assume that the matrix function $\rho_0$ is defined over the entire $\complex$, is smooth and supported in the unit ball. We can do this since we are only looking for local solutions. Moreover, we can suppose that
  \begin{displaymath}
    N = \sup |\rho_0(w)| = \Lnorm{\rho_0}
  \end{displaymath}
  is as small as we want. If we write $g_0 = 1 + f$, then the equation \eqref{eq:gauge_rho_0} becomes
  \begin{equation}\label{eq:gauge_f}
    \frac{\partial f}{\partial \overline{w}} = (1+f) \rho_0.
  \end{equation}
  Define the operator $L$ by
  \begin{equation}
    (L \theta)(w)
      = - \frac{1}{\pi} \iint\limits_{\complex} \frac{\theta(\xi)}{\xi - w} d\mu_\xi
      = \frac{1}{2 \pi i} \iint\limits_{\complex} \frac{\theta(\xi)}{\xi - w} d \xi \wedge d \overline{\xi},
  \end{equation}
  where $\theta$ is any compactly supported matrix-valued function on $\complex$ and $\mu_\xi$ is the Lebesgue measure on $\complex$. Then if $f$ is a solution to the integral equation
  \begin{equation}\label{eq:integral_eqn}
    f = L(\rho_0 + f \rho_0),
  \end{equation}
  it will also satisfy the differential equation \eqref{eq:gauge_f}. The regularity of the elliptic operator $\partialbar$ guarantees that any bounded solution $f$ of the integral equation is smooth, thus we change the problem to solving the integral equation \eqref{eq:integral_eqn} in the Banach space $L^\infty(\complex)$ consisting of bounded matrix-valued functions. By elementary estimation, we have
  \begin{equation}\label{eq:bound_L}
    \Lnorm{L(h \cdot \rho_0)} \leq N \Lnorm{h}, \quad \forall~ h \in L^\infty(\complex).
  \end{equation}
  Thus, when $N < 1$, the map $T_0: L^\infty(\complex) \to L^\infty(\complex)$ defined by $T_0(\phi) = L (\rho_0 + \phi \rho_0)$ is a contraction map. Thus $T_0$ has a unique fixed point $f$, which yields a solution to equation \eqref{eq:gauge_f}. The norm of $f$ is bounded by $\sum_{i \geq 0} N^i \Lnorm{L(\rho_0)}$ and so $g_0 = 1 + f$ can be made invertible if we choose $N$ small enough.

  Now suppose $g_I$ are solutions to those equations in \eqref{eq:gauge_family} on some disc about $0$, with multi-indices $I$ such that $|I| \leq k$, we then want to solve for any given multi-indices $I$ with $|I|=k+1$ the equation
  \begin{equation}\label{eq:gauge_g_I}
    \frac{\partial g_I}{\partial \overline{w}} = g_I \rho_0 + \sum_{\substack{I = J + K  \\ 1 \leq |J| < |I|}} g_J \rho_K.
  \end{equation}
  As before, we define $T_I : L^\infty \to L^\infty$ by $T_I(\phi) = L(\phi \rho_0 + A)$ where $A$ is the summation on the right hand side of equation \eqref{eq:gauge_g_I}. Note that $A$ is determined by solutions with lower multi-indices obtained in previous steps. Moreover, the contraction ratio of $T_I$ is independent of $A$ and only depends on $\rho_0$ in the same way as in \eqref{eq:bound_L}, thus still bounded by $N$. We then apply contraction mapping principle again to get the solution. The key observation is that we do not need to shrink the disc at each step and all solutions $g_I$ live on a common disc with radius corresponding to the constant $N$ fixed at the very beginning. Of course, the upper bounds of the solutions $g_I$ cannot be controlled since they depend on solutions from previous steps, but we do not care about the invertibility of $g_I$ except for $g_0$. So the proof of the one-dimensional case is completed. Higher dimensional case can be proved by exactly the same induction argument as in \cite{Donaldson}. We leave the details to the reader.

\end{proof}

\section{DG-categories and cohesive modules}\label{sec:cohesive}

\subsection{DGAs and dg-categories}

Throughout this section, we use $k$ to denote a fixed field of characteristic 0. For the purpose of the paper, the reader can always take $k = \complex$. We recall the definitions from \cite{Block1}:

\begin{definition}\label{defn:dga}
  A \emph{differential graded algebra}(\emph{dg-algebra} or \emph{dga}) is a pair $$A=( \A^{\bullet}, d)$$ where $\Adot$ is a (non-negatively) graded algebra over $k$, with a derivation
  \begin{displaymath}
    d: \Adot \to \A^{\bullet +1}
  \end{displaymath}
  which satisfies the graded Leibniz rule
  \begin{displaymath}
    d(a \cdot b)=d(a) \cdot b + (-1)^{\vert a \vert} a \cdot d(b).
  \end{displaymath}
  Usually we write $\A$ for $\A^{0}$, the degree zero part of $\Adot$, which can be considered as the ``function algebra" of $A$.
\end{definition}

\begin{remark}
  Almost all the constructions of this section are also available for the more general notion of curved dgas. See \cite{Block1} for details. We will only need to study the case of (uncurved) dgas.
\end{remark}

Recall the definition of dg-categories(\cite{Keller1}, \cite{Keller2}, \cite{Block1}):

\begin{definition}
  A differential graded category (dg-category) over $k$ is a category whose morphisms are $\ZZ$-graded complexes (over $k$) with differentials of degree $+1$. That is, for any two objects $x$ and $y$ in a dg-category $\CC$, the hom set $\CC(x, y)$ forms a $\ZZ$-graded complex of $k$-vector spaces, which we write as $(\CC(x, y), d)$. In addition, the composition map, for $x, y, z \in \Ob \CC$
  \begin{displaymath}
    \CC(x, y) \otimes_k \CC(y, z) \to \CC(x, z)
  \end{displaymath}
  is a morphism of complexes.
\end{definition}

\begin{definition}
  Given a dg-category $\CC$, one can form category $Z^0\CC$ which has the same objects as $\CC$ and whose morphisms are defined by
  \begin{displaymath}
    (Z^0 \CC)(x,y) = Z^0(\CC(x,y)) = \Ker (d: \CC(x,y)^0 \to \CC(x,y)^1),
  \end{displaymath}
  i.e., the closed morphisms of degree zero in $\CC(x,y)$. Similarly, the \emph{homotopy category $\Ho \CC$ of $\CC$} has the same objects as $\CC$ and the morphisms are the 0th cohomology of $\CC(x,y)$:
  \begin{displaymath}
    \Ho \CC(x,y):=H^0(\CC(x,y)).
  \end{displaymath}
  A degree zero closed morphism between objects $x$ and $y$ in a dg-category $\CC$ is said to be a \emph{homotopy equivalence} if it gives an isomorphism in the homotopy category.
\end{definition}

\subsection{The dg-category $\PA$ of a dga}

\subsubsection{The perfect category $\PA$}

\begin{definition}[\cite{Block1}]\label{defn:cohesive}
  For a dga $A=(\Adot, d)$, the associated \emph{perfect category} $\PA$ is a dg-category defined as follows:
  \begin{enumerate}
    \item
      An object $E=(\Edot,\Econn)$ in $\PA$ is called a \emph{cohesive module} over $A$, which consists of two pieces of data: $\Edot$ is a bounded $\ZZ$-graded right module $\Edot$ over $\A$ which is finitely generated and projective and there is a $k$-linear $\ZZ$-connection
        \[\Econn : \Edot \otimes_\A \Adot \to \Edot \otimes_\A \Adot,\]
      which is of total degree one and satisfies the Leibniz condition
        \[\Econn(e \cdot \omega) = \Econn(e \otimes 1)\omega + (-1)^{\vert e \vert}e \cdot d\omega\]
      and the integrability condition
        \[\Econn \circ \Econn  = 0.\]
      Note that such a connection is determined by its value on $\Edot$. Thus we can write $\Econn = \Econn^0 + \Econn^1 + \Econn ^2 + \cdots$, where $\Econn^k : \Edot \to E^{\bullet -k+1} \otimes_\A \A^k$ is the $k$th component of $\Econn$. Then $\Econn^1$ satisfies the Leibniz rule on each $E^n$ while $\Econn^k$ is $\A$-linear for $k \neq 1$.
    \item
      The morphism set $\PA^{\bullet}(E_1,E_2)$ between two cohesive modules $E_1=(E^{\bullet}_1, \Econn_1)$ and $E_2=(E^{\bullet}_2, \Econn_2)$ is a complex, of which the $k$th component $\PA^k(E_1, E_2)$ is defined to be
        \[\{\phi : E^\bullet_1 \otimes_\A \Adot \to E^\bullet_2 \otimes_\A \Adot \vert \deg \phi = k \ \mathrm{and} \ \phi(e \cdot a)=\phi(e)a, \forall a \in \Adot\}.\]
      The differential $d: \PA^{\bullet}(E_1,E_2) \to \PA^{\bullet+1}(E_1,E_2)$ is defined by
        \[d(\phi)(e)=\Econn_2(\phi(e))-(-1)^{\vert\phi\vert}\phi(\Econn_1(e)).\]
      Similarly, a morphism $\phi \in \PA^k(E_1,E_2)$ is determined by its restriction to $E^\bullet_1$ and we denote the components of $\phi$ by
        \[\phi^j : E^\bullet_1 \to E_2^{\bullet+k-j} \otimes_\A \A^j,\]
      which are all $\A$-linear. The composition map
      \begin{displaymath}
        \PA^i(E_2,  E_3) \otimes_k \PA^j(E_1,E_2) \to \PA^{i+j}(E_1,E_3)
      \end{displaymath}
      is defined just by componentwise compositions of $\phi$'s.
  \end{enumerate}
\end{definition}

\begin{proposition}[\cite{Block1}]
  For a dga $A=(\A^\bullet,d)$, the category $\PA$ is a dg-category.
\end{proposition}

\begin{remark}
  In the case when $A = (\AOD(X),\partialbar)$ is the Dolbeault dga of a complex manifold $X$, a cohesive module $(E^\bullet, \Econn) \in \PA$ can be thought of as a complex of smooth vector bundles $(E^\bullet, \Econn^0)$ with a ``twisted holomorphic structure". If the components $\Econn^k$ vanish for all $k \geq 2$, what we get is simply a complex of holomorphic vector bundles with flat $\partialbar$-connections \emph{up to signs} (since we have right modules here).
\end{remark}

\subsubsection{The triangulated structure}

We define a shift functor on the category $\PA$. For $E=(E^\bullet,\Econn)$ set $E[1]=(E[1]^\bullet,\Econn[1])$ where $E[1]^\bullet = E^{\bullet+1}$ and $\Econn[1] = -\Econn$. Then $E[1] \in \PA$. For $E,F \in \PA$ and $\phi \in \PA^0(E,F)$ a closed morphism of degree $0$, define the cone of $\phi$, $\Cone(\phi)=(\Cone(\phi)^\bullet, \Cconn_\phi)$ by
\begin{displaymath}
  \Cone(\phi)^\bullet =
    \begin{pmatrix}
      F^\bullet \\
      \oplus  \\
      E[1]^\bullet
    \end{pmatrix}
\end{displaymath}
and
\begin{displaymath}
  \Cconn_\phi =
    \begin{pmatrix}
      \Fconn & \phi  \\
      0 & \Econn[1]
    \end{pmatrix}
\end{displaymath}
We then have a triangle of degree $0$ closed morphisms
\begin{equation}\label{eq:cone_PA}
  E \xrightarrow{\phi} F \to \Cone(\phi) \to E[1].
\end{equation}

\begin{proposition}[Prop. 2.7., \cite{Block1}]
  Let $A$ be a dga. The dg-category $\PA$ is pretriangulated in the sense of Bondal and Kapranov, \cite{BondalKapranov}. In particular, the homotopy category $\Ho\PA$ is triangulated and its triangulated structure is given by distinguished triangles which are isomorphic to those of the form \eqref{eq:cone_PA}.
\end{proposition}

\subsubsection{Homotopy equivalences}

As defined in the case of general dg-categories, a closed morphism $\phi$ of degree $0$ between cohesive modules $E_i=(E^\bullet_i,\Econn_i)$, $i=1,2$, over $A$ is a \emph{homotopy equivalence} if it induces an isomorphism in $\Ho\PA$. We have the following simple criterion for $\phi$ to be a homotpy equivalence:

\begin{proposition}[Prop. 2.9., \cite{Block1}]
  Let $\phi \in \PA^0(E_1,E_2)$ be a closed morphism of degree zero. Then $\phi$ is a homotopy equivalence if and only if $\phi^0: (E^\bullet_1,\Econn^0_1) \to (E^\bullet_2,\Econn^0_2)$ is a quasi-isomorphism of complexes of $\A$-modules.
\end{proposition}

\subsection{The case of complex manifolds}\label{subsec:cohesive_complexmnfld}

Let $X$ be a smooth compact complex manifold with the structure sheaf $\Osheaf_X$ of holomorphic functions on $X$. Analogous to the study of algebraic geometry, we want to consider the derived category of $\Osheaf_X$-modules. However, here is a subtlety about the definition of derived category in the world of analytical geometry. Of course, one can define as usual the bounded derived category $\Dcat(\Coh X)$ of complexes of coherent sheaves of $\Osheaf_X$-modules. On the other hand, there is the bounded derived category $\Dcoh(X)$ of complexes of sheaves of $\Osheaf_X$-modules with coherent cohomology. In the case of a noetherian scheme, there two categories are equivalent. For a complex analytic space it is not known if they are equivalent or not. For our purpose, the category $\Dcoh(X)$ is more flexible to work with, thus we will only use this version of derived category from now on.

Let $A=(\AOD(X),\partialbar)$ be the Dolbeault dga of $X$. The following theorem of Block (Theorem 4.3., \cite{Block1}) states that the perfect category $\PA$ associated to the Dolbeault dga provides a dg-enhancement of the derived category $\Dcoh(X)$.

\begin{theorem}\label{thm:Block_complex_manifold}
  Let $X$ be a compact complex manifold and $A=(\AOD(X),\partialbar)$ its Dolbeault dga. Then the homotopy category $\Ho \PA$ of the dg-category $\PA$ is equivalent to $\Dcoh(X)$, the bounded derived category of complexes of sheaves of $\Osheaf_X$-modules with coherent cohomology.
\end{theorem}

We will postpone the proof to \S \ref{subsec:cohesive_formalnbhd}, where we shall prove a similar result for the Dolbeault dga of a formal neighborhood, which covers the present theorem as a special case. The strategy of the proof is almost the same as that in \S ~ 4, \cite{Block1}, except for the subtlety that the proof in \cite{Block1} makes use of the fact that the sheaf of Dolbeault dga on a usual complex manifold is flat over the sheaf of holomorphic functions (Theorem. \ref{thm:Malgrange}), while such conclusion in the case of formal neighborhood is not known to us and more sophisticated results from \S ~ \ref{subsec:Dolbeault_resolution} are needed.

\subsection{Descent of $\PA$}\label{subsec:descent}

We briefly recall the main results in \cite{Ben-Bassat-Block} concerning the gluing of dg-categories of the form $\PA$.

\subsubsection{Inverse image functor}\label{subsubsec:Descent_Functoriality}

We define the inverse image functor between categories of the form $\PA$. Given two dgas, $A_i=(\A^\bullet_i, d_i)$, $i=1,2$, and a homomorphism $f: \Adot_1 \to \Adot_2$, we define a dg functor
\begin{displaymath}
  f^*: \PP_{A_1} \to \PP_{A_2}
\end{displaymath}
as follows. Given $E=(E^\bullet,\Econn)$ a cohesive module over $A_1$, set $f^*(E)$ to be the cohesive module over $A_2$
\begin{displaymath}
  (E^\bullet \otimes_{\A_1} \A_2, \Econn_2),
\end{displaymath}
where the connection $\Econn_2$ is defined by
  \[ \Econn_2(e \otimes b) = \Econn(e) b + (-1)^{|e|} e \otimes d_2 b. \]
Then $\Econn_2$ is still a $\ZZ$-connection and satisfies $(\Econn_2)^2 = 0$. Given $E=(E^\bullet,\Econn)$ and $F=(F^\bullet,\Fconn)$ in $\PP_{A_1}$ and a morphism $\phi \in \PP_{A_1}^\bullet(E,F)$, the morphism $f^* \phi \in \PP^\bullet_{A_1}(f^*E, f^*F)$ is determined by the composition of maps
\begin{displaymath}
  E^\bullet \xrightarrow{~\phi~} F^\bullet \otimes_{\A_1} \Adot_1 \to F^\bullet \otimes_{\A_1} \Adot_2.
\end{displaymath}
The functor defined above takes objects in $\PP_{A_1}$ to objects in $\PP_{A_2}$. Given two homomorphisms $\homof: A_1 \to A_2$ and $\homog: A_2 \to A_3$, there are natural equivalences $(\homog \circ \homof)^* \Longrightarrow \homog^* \circ \homof^*$ which satisfy the obvious coherence relation.

\begin{remark}
An analogue of the direct image functor is in general not defined for $\PA$, but for larger dg-categories. For details and more general construction of functors between perfect categories and their cousins (such as \emph{quasi-perfect categories}), we refer the reader to \cite{Block1} and \cite{Ben-Bassat-Block}. We will not need them explicitly in the statements of the results below from \cite{Ben-Bassat-Block}.
\end{remark}

\subsubsection{Homotopy fiber products of dg-categories}

We recall the notion of homotopy fiber products of dg-categories from \S 4, \cite{Ben-Bassat-Block}. Let $B$, $C$ and $D$ be dg-categories, together with dg-functors $L: C \to D$ and $G: B \to D$. We then have the homotopy cartesian square of dg-categories
\begin{diagram}
  B \times^h_D C  & \rTo^{p_B}  & B  \\
  \dTo_{p_C}      &             & \dTo_{G}  \\
  C               & \rTo^{L}    & D.
\end{diagram}
The homotopy fiber product $B \times^h_D C$ is a dg-category whose objects are triples $(M, N, \phi)$ where $M \in B$, $N \in C$ and $\phi \in D^0(G(M),L(N))$, such that $\phi$ is closed and becomes invertible in $\Ho D$, i.e., a homotopy equivalence between $G(M)$ and $L(N)$. The morphisms are given by the complex
\begin{displaymath}
  \begin{split}
  (B \times^h_D C)^\bullet &((M_1,N_1,\phi_1),(M_2,N_2,\phi_2)) \\
   &:= B^\bullet(M_1,M_2) \oplus C^\bullet(N_1,N_2) \oplus D^{\bullet-1}(G(M_1),L(N_2))
  \end{split}
\end{displaymath}
with the differential given by
\begin{displaymath}
  d (\mu, \nu, \gamma) = (d\mu, d\nu, d\gamma + \phi_2 G(\mu) - (-1)^i L(\nu) \phi_1),
\end{displaymath}
where $i = \deg \mu = \deg \nu$. The composition
\begin{displaymath}
  \begin{split}
    & (B \times^h_D C)^\bullet ((M_2,N_2,\phi_2),(M_3,N_3,\phi_3)) \otimes (B \times^h_D C)^\bullet ((M_1,N_1,\phi_1),(M_2,N_2,\phi_2))  \\
    & \to (B \times^h_D C)^\bullet ((M_1,N_1,\phi_1),(M_3,N_3,\phi_3))
  \end{split}
\end{displaymath}
is given by
\begin{displaymath}
  (\mu',\nu',\gamma') \circ (\mu,\nu,\gamma) = (\mu'\mu, \nu'\nu, \gamma' G(\mu) + L(\nu') \gamma).
\end{displaymath}

\begin{lemma}[Lemma 4.1., \cite{Ben-Bassat-Block}]
  A closed morphism
    \[ (\mu,\nu,\gamma) \in (B \times^h_D C)^0 ((M_1,N_1,\phi_1),(M_2,N_2,\phi_2)) \]
  is a homotopy equivalence if and only if $\mu$ and $\nu$ are homotopy equivalences in $B$ and $C$ respectively.
\end{lemma}

\subsubsection{Descent}

Let $A=(\A^\bullet,d_A)$, $B=(\B^\bullet,d_B)$, $C=(\C^\bullet,d_C)$, $D=(\D^\bullet,d_D)$ be dgas with homomorphisms of dgas forming a commutative diagram
\begin{diagram}\label{diag:homo_dgas}
  A                 & \rTo^{f}      & B  \\
  \dTo_{k}     &                    & \dTo_{g}  \\
  C                 & \rTo^{l}      & D.
\end{diagram}
We will be interested in the homotopy fiber product of dg-categories $\hprodperf$ corresponding to the diagram
\begin{diagram}
  &       &                  & \PP_{B} \\
  &       &                  & \dTo_{g^*} \\
  &\PP_C  & \rTo^{l^*}  & \PP_D.
\end{diagram}
There is a dg-functor
\begin{equation}\label{eq:restriction}
  R: \PA \to \hprodperf
\end{equation}
defined by
\begin{equation}\label{eq:restriction_defn}
  R(S) = (f^* S, k^* S, \kappa),
\end{equation}
where $\kappa$ is the canonical isomorphism $g^* f^* S \to l^* k^* S$. This can be thought of as the restriction functor.

Before stating the main theorem (Theorem 6.7.), we need several assumptions for the theorem to be true.

\begin{assumption}\label{assumption}
  \begin{enumerate}[(1)]
    \item
      the degree $0$ components $f^0$, $g^0$, $k^0$, $l^0$ form a fiber product:
      \begin{displaymath}
        \A^0 \simeq \B^0 \times_{\D^0} \C^0 ;
      \end{displaymath}
    \item
      $l^0: \C^0 \to \D^0$ is surjective;
    \item
      $\Adot$ is flat over $\A^0$
    \item
      the natural maps induced by $f$, $g$, $k$, $l$ induce isomorphisms
      \begin{displaymath}
        \B^0 \otimes_{\A^0} \Adot \to \B^\bullet, ~ \C^0 \otimes_{\A^0} \Adot \to \C^\bullet,
        ~ \D^0 \otimes_{\B^0} \B^\bullet \to \D^\bullet, ~ \D^0 \otimes_{\C^0} \C^\bullet \to \D^\bullet.
      \end{displaymath}
  \end{enumerate}
\end{assumption}

\begin{theorem}[Theorem 6.7., \cite{Ben-Bassat-Block}]\label{thm:descent_perf}
  Let $A$, $B$, $C$, $D$ be dgas with homomorphisms as in the diagram \eqref{diag:homo_dgas} satisfying Assumption \ref{assumption}. Then the restriction functor $R$ is a dg-quasi-equivalence between the dg-categories $\PA$ and $\hprodperf$.
\end{theorem}

\subsubsection{The case of formal neighborhoods}\label{subsubsec:descent_formalnbhd}

We apply here Theorem \ref{thm:descent_perf} to the Dolbeault dgas of formal neighborhoods as defined in \S ~ \ref{sec:Dolbeault_DGA}. Let $Y$ be a complex manifold and $X \subset Y$ be a closed complex submanifold of $Y$. We have the Dolbeault dga $A = (\Adot(\Yhat),\partialbar)$ and its sheafy version $\Asheaf^\bullet_{\Yhat}$ as a sheaf of dgas over $X$. For any closed subset $Z \subseteq X$, we define the dga $A(Z) = (\Adot(Z),\partialbar)$ by
\begin{displaymath}
  \Adot(Z) = \varinjlim_{V} \Gamma(V,\Asheaf^\bullet_{\Yhat}) = \Gamma(Z, \Asheaf^\bullet_{\Yhat}|_Z),
\end{displaymath}
where the direct limit is taken over all open subsets $V$ of $X$ containing $Z$ and $\Asheaf^\bullet_{\Yhat}|_Z$ is the sheaf over $Z$
\begin{displaymath}
  \Asheaf^\bullet_{\Yhat}|_Z = i_Z^{-1} \Asheaf^\bullet_{\Yhat},
\end{displaymath}
where $i_Z: Z \hookrightarrow X$ is the inclusion. For an inclusion of closed subsets $Z_1 \subseteq Z_2$, there is a natural homomorphism of dgas $A(Z_2) \to A(Z_1)$. The corresponding inverse image functor $\PP_{A(Z_2)} \to \PP_{A(Z_1)}$ as constructed in \S ~ \ref{subsubsec:Descent_Functoriality} will be denoted by $E \mapsto E|_{A(Z_1)}$.

\begin{theorem}\label{thm:descent_formalnbhd}
  Let $Y$ be a complex manifold and $X$ be a closed complex submanifold of $Y$. Suppose $Z_1$ and $Z_2$ are two closed subsets of $X$. Then the natural restriction functor
  \begin{displaymath}
    \PP_{A(Z_1 \cup Z_2)} \to \PP_{A(Z_1)} \times^h_{\PP_{A(Z_1 \cap Z_2)}} \PP_{A(Z_2)}
  \end{displaymath}
  is a dg-quasi-equivalence.
\end{theorem}

\begin{proof}
  The proof is exactly the same as that of Theorem 7.4., (a), \cite{Ben-Bassat-Block}.
\end{proof}

\subsection{Cohesive modules over a formal neighborhood}\label{subsec:cohesive_formalnbhd}

We will prove in this section that the category $\PA$ associated to the Dolbeault dga $A=(\Adot(\Yhat),\partialbar)$ of a formal neighborhood is a dg-enhancement of the derived category $\Dcoh(\Yhat)$, under the assumption that the submanifold $X$ is compact.

Let $X$ be a compact complex submanifold of a complex manifold $Y$. Recall that we have constructed the Dolbeault dga $A=(\Adot(\Yhat),\partialbar)$ of the formal neighborhood $\Yhat$ of $X$ inside $Y$. We also have the bounded derived category $\Dcoh(\Yhat)$ of complexes of sheaves of $\Osheaf_{\Yhat}$-modules with coherent cohomology. Since we are in the smooth situation, $\Dcoh(\Yhat)$ is equivalent of the derived category $D_{perf}(\Yhat)$ of perfect complexes of $\Osheaf_{\Yhat}$-modules. We will prove in this section the following analogue of Theorem \ref{thm:Block_complex_manifold} for $\Yhat$.

\begin{theorem}\label{thm:PA_Yhat}
  Suppose $X \hookrightarrow Y$ is a closed embedding of complex manifolds and $X$ is compact. Let $A=(\Adot(\Yhat),\partialbar)$ be the Dolbeault dga of the formal neighborhood $\Yhat$, then the category homotopy category $\Ho \PA$ of the dg-category $\PA$ is equivalent to $\Dcoh(\Yhat)$, the bounded derived category of complexes of sheaves of $\Osheaf_{\Yhat}$-modules with coherent cohomology.
\end{theorem}

The proof is broken up into several steps. First we construct a functor
\begin{displaymath}
  \alpha: \Ho \PA \to \Dcoh(\Yhat)
\end{displaymath}
as follows. Recall that $\Asheaf^\bullet_{\Yhat}$ is the fine sheaf of Dolbeault dgas such that its global sections are $\Gamma(X,\Asheaf^\bullet_{\Yhat}) = \Adot(\Yhat)$. Then a module $M$ over $\A(\Yhat)$ naturally localizes to a sheaf $\Msheaf$ of $\Asheaf(\Yhat)$-module over $X$, such that the sections over an open subset $U$ are
\begin{displaymath}
  \Msheaf(U) = M \otimes_{\A(\Yhat)} \Asheaf_{\Yhat}(U).
\end{displaymath}
For an object $E=(E^\bullet,\Econn)$ of $\PA$, define the sheaves $\Esheaf^{p,q}$ by
  \[ \Esheaf^{p,q}(U) = E^p \otimes_{\A(\Yhat)} \Asheaf_{\Yhat}(U). \]
We define a complex of sheaves by $(\Esheaf^\bullet,\Econn) = (\sum_{p+q=\bullet} \Esheaf^{p,q}, \Econn)$. This a complex of fine sheaves of $\Osheaf_{\Yhat}$-modules, since $\Econn$ is a $\partialbar$-connection.

\begin{lemma}
  The complex $(\Esheaf^\bullet, \Econn)$ defined above has coherent cohomology. Moreover,
    \[ E=(E^\bullet,\Econn) \mapsto \alpha(E)=(\Esheaf^\bullet,\Econn) \]
  defines a fully faithful functor $\alpha: \Ho \PA \to \Dcoh(\Yhat)$.
\end{lemma}

\begin{proof}
  The proof is exactly the same as that of Lemma 4.5., \cite{Block1}. The main ingredient of the proof is to show that, for any point $x \in X$, there exits a polydisc $U$ containing $x$ and a gauge transformation $g: \Esheaf^\bullet|_V \xrightarrow{\simeq} \Esheaf^\bullet|_V$ of degree zero such that $g \circ \Econn \circ g^{-1} = \Fconn^0 + \partialbar$. In other words, we want to show that $\Esheaf^\bullet|_V$ is gauge equivalent to the Dolbeault resolution of a complex of holomorphic vector bundles and hence quasi-isomorphic to a complex of locally free $\Osheaf_{\Yhat}$-modules as discussed in \S ~ \ref{subsec:hol_vector_bundle}, which therefore has coherent cohomology. The proof of the existence of such gauge transformations for general $E=(E^\bullet,\Econn)$ is essentially the same as for the case of a single vector bundle as established by Theorem \ref{thm:vectorbundle-integrable}, following the inductive argument used in \cite{Block1}. The details are left to the reader. The assertion that $\alpha$ is fully faithful follows from the fact that $\Asheaf^\bullet_{\Yhat}$ is fine and the definition of the derived functor $\gext$.
\end{proof}

The following proposition completes the proof of Theorem \ref{thm:PA_Yhat}.

\begin{proposition}\label{prop:alpha_surjectivity}
  With the notations above, to any complex of sheaves of $\Osheaf_{\Yhat}$-modules $(\Msheaf^\bullet,d)$ on $X$ with coherent cohomology there corresponds a cohesive module $E=(E^\bullet,\Econn)$ over the Dolbeault dga $A=(\A^\bullet(\Yhat),\partialbar)$, unique up to quasi-isomorphism in $\PA$, such that $\alpha(E)$ is isomorphic to $(\Msheaf^\bullet,d_{\Msheaf})$ in $\Dcoh(\Yhat)$. In other words, the functor $\alpha: \Ho \PA \to \Dcoh(\Yhat)$ is essentially surjective.
\end{proposition}

This is an analogue to Lemma 4.6., \cite{Block1}. However, the proof here is more complicated, since the original proof in \cite{Block1} relies on the flatness of the sheaf of Dolbeault dgas $\Asheaf^{0,\bullet}_X$ over $\Osheaf_X$ of a complex manifold (Theorem \ref{thm:Malgrange}). However, we do not have the flatness of $\Asheaf_{\Yhat}$ over $\Osheaf_{\Yhat}$ as mentioned in \S ~ \ref{subsec:Dolbeault_resolution}. Moreover, even given the exactness of the functor $~ \dash \otimes_{\Osheaf_{\Yhat}} \Asheaf^\bullet_{\Yhat}$ on $\Coh(\Yhat)$ (Theorem \ref{thm:Dolbeault_main}), we cannot directly tensor a given complex $\Msheaf^\bullet \in \Dcoh(\Yhat)$ with $\Asheaf^\bullet_{\Yhat}$, since the components of a complex $\Msheaf^\bullet$ are not necessarily coherent $\Osheaf_{\Yhat}$-modules on the nose and we do not know if every $\Msheaf^\bullet$ is quasi-isomorphic to a complex of coherent $\Osheaf_{\Yhat}$-modules in the analytic setting (see the discussion at the beginning of \S ~ \ref{subsec:cohesive_complexmnfld}).

To circumvent these issues, we notice that one can always find for $\Msheaf^\bullet$ quasi-isomorphic complexes of sheaves of free $\Osheaf_{\Yhat}$-modules locally (say, over Stein subsets). By tensoring them with the Dolbeault dga, we obtain cohesive modules defined locally and then glue them together to get a cohesive module globally defined using the descent results from \S ~ \ref{subsec:cohesive_complexmnfld}.

The precise argument goes as follows. For any closed subset $Z \subseteq X$, we define the dg-category $\Cperf(Z)$ whose objects are perfect complexes of sheaves of $\Osheaf_{\Yhat}|_Z$-modules over $Z$ with morphisms being $\Osheaf_{\Yhat}|_Z$-linear maps between complexes. Here $\Osheaf_{\Yhat}|_Z = i^{-1}_Z \Osheaf_{\Yhat}$ with $i_Z: Z \hookrightarrow X$ being the inclusion. Then we have natural restriction functors $\Cperf(Z_1) \to \Cperf(Z_2)$ for $Z_2 \subseteq Z_1$. Moreover, adopting the notations from \S ~ \ref{subsubsec:descent_formalnbhd}, there is a natural functor $\alpha_Z : \PP_{A(Z)} \to \Cperf(Z)$ defined for any closed subset $Z$ by the same localizing procedure as above, such that the diagram of dg-functors
\begin{diagram}[size=2.2em]
  \PP_{A(Z_1)}          & \rTo  & \PP_{A(Z_2)}  \\
  \dTo^{\alpha_{Z_1}}   &       & \dTo_{\alpha_{Z_2}}  \\
  \Cperf(Z_1)           & \rTo  & \Cperf(Z_2)
\end{diagram}
commutes (up to natural equivalence) for any $Z_2 \subseteq Z_1$. Note that $\alpha_Z$ sends homotopy equivalences in $\PP_{A(Z)}$ to quasi-isomorphisms of complexes of sheaves in $\Cperf(Z)$.

Given two closed subsets $Z_1, Z_2 \subseteq X$, denote the natural inclusions of closed sets by
\begin{diagram}[width=2.4em]
  Z_1 \cup Z_2     & \lTo^{\quad i_1}   & Z_1  \\
  \uTo^{i_2}       &              & \uTo_{j_1}  \\
  Z_2              & \lTo^{j_2}   & Z_1 \cap Z_2.
\end{diagram}
Then we can define the restriction functor
  \[ \Cperf(Z_1 \cup Z_2) \xrightarrow{~R~} \hprodcperf \]
by formula similar to \eqref{eq:restriction_defn}:
  \[ R(\Fsheaf^\bullet) = (i^{-1}_1 \Fsheaf^\bullet, i^{-1}_2 \Fsheaf^\bullet, \kappa) = (\Fsheaf^\bullet|_{Z_1}, \Fsheaf^\bullet|_{Z_2}, \kappa), \]
where $\kappa: j^{-1}_1 i^{-1}_1 \Fsheaf^\bullet \to j^{-1}_2 i^{-1}_2 \Fsheaf^\bullet$ is the canonical isomorphism.

\begin{lemma}\label{lemma:comm_diagram_perfect}
  We have the commutative diagrams (up to natural equivalences)
  \begin{diagram}[height=2.5em]
    \PP_{A(Z_1 \cup Z_2)}          & \rTo^{R}  & \PP_{A(Z_1)} \times^h_{\PP_{A(Z_1 \cap Z_2)}} \PP_{A(Z_2)}  \\
    \dTo^{\alpha_{Z_1 \cup Z_2}}   &           & \dTo_{\alpha_{Z_1} \times \alpha_{Z_2}}  \\
    \Cperf(Z_1 \cup Z_2)           & \rTo^{\quad R \quad}  & \hprodcperf.
  \end{diagram}
\end{lemma}

There is also a right adjoint functor of $R$
  \[ \Cperf(Z_1 \cup Z_2) \xleftarrow{~A~} \hprodcperf \]
which is defined as follows: given $(\Msheaf^\bullet, \Nsheaf^\bullet, \phi) \in \hprodcperf$, consider the closed morphism
\begin{equation}
  \lambda \in \Cperf(Z_1 \cup Z_2)^0 (i_{1*} \Msheaf^\bullet \oplus i_{2*} \Nsheaf^\bullet, i_{1*} j_{1*} j_2^{-1} \Nsheaf^\bullet)
\end{equation}
defined as the difference of the maps
\begin{displaymath}
  i_{1*} \Msheaf^\bullet \to i_{1*} j_{1*} j_1^{-1} \Msheaf^\bullet \xrightarrow{i_{1*} j_{1*} \phi} i_{1*} j_{1*} j_2^{-1} \Nsheaf^\bullet
\end{displaymath}
and
\begin{equation}\label{eq:glue_cone}
  i_{2*} \Nsheaf^\bullet \to i_{2*} j_{2*} j_2^{-1} \Nsheaf^\bullet = i_{1*} j_{1*} j_2^{-1} \Nsheaf^\bullet,
\end{equation}
where the maps $i_{1*} \Msheaf^\bullet \to i_{1*} j_{1*} j_1^{-1} \Msheaf^\bullet$ and $i_{2*} \Nsheaf^\bullet \to i_{2*} j_{2*} j_2^{-1} \Nsheaf^\bullet$ are induced by the adjunction between the pushforward and pullback functors. We then define the dg-functor
  \[ A: \hprodcperf \to \Cperf(Z_1 \cup Z_2) \]
by
  \[ A(\Msheaf^\bullet, \Nsheaf^\bullet, \phi) = \Cone(\lambda)[-1]. \]
To define the maps of hom complexes, for any morphism
$(\mu,\nu,\gamma)$ in $\hprodcperf$ from $(\Msheaf^\bullet_1, \Nsheaf^\bullet_1, \phi_1)$ to $(\Msheaf^\bullet_2, \Nsheaf^\bullet_2, \phi_2)$, we need to define a corresponding morphism in $\Cperf(Z_1 \cup Z_2) (\Fsheaf^\bullet_1, \Fsheaf^\bullet_2)$ where $\Fsheaf^\bullet_p = A(\Msheaf^\bullet_p, \Nsheaf^\bullet_p, \phi_p)$, $p=1,2$, as described above. Using the representations
\begin{displaymath}
  \Fsheaf^\bullet_p =
    \begin{pmatrix}
      i_{1*} j_{1*} j_2^{-1} \Nsheaf^{\bullet-1}_p \\
      \oplus \\
      i_{1*} \Msheaf^\bullet_p \oplus i_{2*} \Nsheaf^\bullet_p
    \end{pmatrix},
\end{displaymath}
we can express the corresponding map as
\begin{displaymath}
  \begin{pmatrix}
    i_{1*} j_{1*} j_2^{-1} \nu    & \epsilon  \\
    0                             & i_{1*}(\mu) \oplus i_{2*}(\nu)
  \end{pmatrix}
  : \Fsheaf^\bullet_1 \to \Fsheaf^\bullet_2
\end{displaymath}
where $\epsilon$ is the composition
\begin{displaymath}
  i_{1*} \Msheaf^\bullet_1 \oplus i_{2*} \Nsheaf^\bullet_1 \to i_{1*} \Msheaf^\bullet_1 \to i_{1*} j_{1*} j_1^{-1} \Msheaf^\bullet_1 \xrightarrow{i_{1*} j_{1*} \gamma} i_{1*} j_{1*} j_2^{-1} \Nsheaf^{\bullet-1}_2.
\end{displaymath}
Moreover, there is a natural isomorphism of complexes
\begin{displaymath}
  (\hprodcperf)^\bullet (R(\Fsheaf^\bullet), (\Msheaf^\bullet, \Nsheaf^\bullet, \phi)) \simeq \Cperf(Z_1 \cup Z_2)^\bullet( \Fsheaf^\bullet, A(\Msheaf^\bullet, \Nsheaf^\bullet, \phi))
\end{displaymath}
which is natural in both variables $\Fsheaf^\bullet$ and $(\Msheaf^\bullet, \Nsheaf^\bullet, \phi)$. Thus $A$ is a right adjoint to $R$. The construction of the functor $A$ here is in fact a counterpart of the one for quasi-perfect categories in Definition 6.2., \cite{Ben-Bassat-Block}.

\begin{lemma}\label{lemma:adjunction_quasi-isom}
  The natural adjunction
    \[ \Id_{\Cperf(Z_1 \cup Z_2)} \Longrightarrow A \circ R \]
  gives a quasi-isomorphism for each object in $\Cperf(Z_1 \cup Z_2)$.
\end{lemma}

\begin{proof}
  For any inclusion of closed subset $i: Z \hookrightarrow Z_1 \cup Z_2$ and any sheaf $\Fsheaf$ over $Z_1 \cup Z_2$, we set
    \[ \Fsheaf_{Z} = i_* i^{-1} \Fsheaf. \]

  Given any $\Fsheaf^\bullet \in \Cperf(Z_1 \cup Z_2)$, we need to show that the natural map
    \[ \Fsheaf^\bullet \to A \circ R(\Fsheaf^\bullet) = A (\Fsheaf^\bullet|_{Z_1}, \Fsheaf^\bullet|_{Z_2}, \kappa) \]
  is a quasi-isomorphism of complexes of sheaves. Note that the map
    \[ \lambda: i_{1*} i_1^{-1} \Fsheaf^\bullet \oplus i_{2*} i_2^{-1} \Fsheaf^\bullet \to i_{1*} j_{1*} j_2^{-1} i_2^{-1} \Fsheaf^\bullet \]
  or
    \[ \lambda: \Fsheaf^\bullet_{Z_1} \oplus \Fsheaf^\bullet_{Z_2} \to \Fsheaf^\bullet_{Z_1 \cap Z_2} \]
  is surjective, thus $A \circ R(\Fsheaf^\bullet)$ is quasi-isomorphic to $\Ker \lambda$. On the other hand, we have $\Ker \lambda \simeq \Fsheaf^\bullet = \Fsheaf^\bullet_{Z_1 \cup Z_2}$ due to the short exact sequence
    \[ 0 \to \Fsheaf^\bullet_{Z_1 \cup Z_2} \xrightarrow{\alpha} \Fsheaf^\bullet_{Z_1} \oplus \Fsheaf^\bullet_{Z_2} \xrightarrow{\lambda} \Fsheaf^\bullet_{Z_1 \cap Z_2} \to 0, \]
  where $\alpha=(\alpha_1,\alpha_2)$ is induced by the natural morphisms $\alpha_p: \Fsheaf_{Z_1 \cup Z_2} \to \Fsheaf_{Z_p}$, $p=1, 2$ (Prop. 2.3.6., (vi), p. 94, \cite{Kashiwara}).
\end{proof}

\begin{lemma}\label{lemma:A_quasi-isom}
  Let
    \[ (\mu, \nu, \gamma) \in (\hprodcperf)^0((\Msheaf^\bullet_1,\Nsheaf^\bullet_1,\phi_1), (\Msheaf^\bullet_2,\Nsheaf^\bullet_2,\phi_2)) \]
  be a degree $0$ closed morphism, such that $\mu: \Msheaf^\bullet_1 \to \Msheaf^\bullet_2$ and $\nu: \Nsheaf^\bullet_1 \to \Nsheaf^\bullet_2$ are quasi-isomorphisms of complexes of sheaves over $Z_1$ and $Z_2$ respectively. Then the functor $A$ sends $(\mu, \nu, \gamma)$ to a quasi-isomorphism of complexes of sheaves in $\Cperf(Z_1 \cup Z_2)$.
\end{lemma}

\begin{proof}
  This can be deduced from the short exact sequences of complexes of sheaves
    \[ 0 \to i_{1*} j_{1*} j_2^{-1} \Nsheaf^\bullet_p [-1] \to \Cone(\lambda_p)[-1] \to i_{1*} \Msheaf^\bullet_p \oplus i_{2*} \Nsheaf^\bullet_p \to 0 \]
  ($p=1, 2$) and the associated long exact sequences of cohomology sheaves, where
  \begin{displaymath}
    \lambda_p \in \Cperf(Z_1 \cup Z_2)^0 (i_{1*} \Msheaf^\bullet_p \oplus i_{2*} \Nsheaf^\bullet_p, i_{1*} j_{1*} j_2^{-1} \Nsheaf^\bullet_p), \quad p = 1, 2,
  \end{displaymath}
  are as defined above.
\end{proof}

\begin{proof}[Proof of Proposition \ref{prop:alpha_surjectivity}]
  Since $X$ is compact, we can cover it by a finite collection of Stein compacts $\{ K_i \}_{i=1}^{r}$. Let
    \[ W_s = \bigcup_{j=1}^{s} K_j , \quad 1 \leq s \leq r, \]
  so that $W_r = X$. We prove by induction that, given any perfect complex of $\Osheaf_{\Yhat}$-modules $(\Msheaf^\bullet,d_{\Msheaf}) \in \Cperf(X)$, for each $s$ there exits $E_s = (E^\bullet_s, \Econn_s) \in \PP_{A(W_s)}$ such that $\alpha_{W_s}(E_s)$ is isomorphic to $\Msheaf^\bullet|_{W_s}$ in $D_{pe}(W_s)$, where $D_{pe}(W_s)$ denotes the bounded derived category of perfect complexes of sheaves of $\Osheaf_{\Yhat}|_{W_s}$-modules over $W_s$. In particular, this implies the proposition for $D_{pe}(W_r) \simeq \Dcoh(\Yhat)$.

  We first check the assertion for $W_1 = K_1$. Since $K_1$ is Stein compact and $(\Msheaf^\bullet,d_{\Msheaf})$ is a perfect complex, we can find a bounded complex of free $\Osheaf_{\Yhat}|_{K_1}$-modules of finite type $(\Esheaf^\bullet,d_{\Esheaf})$ and a quasi-isomorphism of complexes of sheaves of $\Osheaf_{\Yhat}|_{K_1}$-modules
    \[ (\Esheaf^\bullet,d_{\Esheaf}) \to (\Msheaf^\bullet|_{K_1},d_{\Msheaf}). \]
  In particular, $(\Esheaf^\bullet,d_{\Esheaf})$ is a complex of coherent $\Osheaf_{\Yhat}|_{K_1}$-modules, thus by Theorem \ref{thm:Dolbeault_resolution_Yhat} the natural inclusion
    \[ (\Esheaf^\bullet,d_{\Esheaf}) \hookrightarrow (\Esheaf^\bullet \otimes_{\Osheaf_{\Yhat}|_{K_1}} \Asheaf^\bullet_{\Yhat}|_{K_1}, d_{\Esheaf} \otimes 1 + 1 \otimes \partialbar) \]
  is a quasi-isomorphism of complexes of $\Osheaf_{\Yhat}|_{K_1}$-modules. Taking global sections of the latter gives a cohesive module
    \[ E_1 = (E^\bullet_1, \Econn_1) = (\Esheaf^\bullet(K_1) \otimes_{\Osheaf_{\Yhat}(K_1)} \Asheaf_{\Yhat}(K_1), d_{\Esheaf} \otimes 1 + 1 \otimes \partialbar) \]
  over $A(K_1) = A(W_1)$, which satisfies the desired property.

  Suppose now the assertion has been proven for some given $s$, we need to show that it is also true for $s+1$. By the inductive assumption, there exits $E_s = (E^\bullet_s, \Econn_s) \in \PP_{A(W_s)}$ such that $\alpha_{W_s}(E_s)$ is isomorphic to $\Msheaf^\bullet|_{W_s}$ in $D_{pe}(W_s)$, which means that there is a complex $\Nsheaf^\bullet \in \Cperf(W_s)$ and quasi-isomorphisms of complexes of sheaves
    \[ \Msheaf^\bullet|_{W_s} \xleftarrow{\nu} \Nsheaf^\bullet \xrightarrow{\sigma} \alpha_{W_s}(E_s). \]

  As in the case $s=1$ as above, we can always find over the Stein compact $K_{s+1}$ a bounded complex of free $\Osheaf_{\Yhat}|_{K_{s+1}}$-modules of finite type $(\Esheaf^\bullet,d_{\Esheaf})$ and a quasi-isomorphism of complexes of sheaves of $\Osheaf_{\Yhat}|_{K_{s+1}}$-modules
    \[ \mu: (\Esheaf^\bullet,d_{\Esheaf}) \to (\Msheaf^\bullet|_{K_{s+1}},d_{\Msheaf}). \]
  By tensoring $\Esheaf^\bullet$ with $\Asheaf_{\Yhat}$ and taking the global sections as before, we obtain a cohesive module $E = (E^\bullet,\Econn) \in \PP_{A(K_{s+1})}$ and a natural inclusion map $\rho: \Esheaf^\bullet \hookrightarrow \alpha_{K_{s+1}} (E)$.

  Set $Z = K_{s+1} \cap W_s$, then again thanks to the freeness of $\Esheaf^\bullet$, there exists a degree $0$ closed morphism $\phi \in \Cperf(Z)^0(\Esheaf^\bullet|_{Z}, \Nsheaf^\bullet|_{Z})$ which completes
  \begin{diagram}[size=2.6em]
    \Esheaf^\bullet|_{Z}  & \rDashto^{~\phi~}  & \Nsheaf^\bullet|_{Z}  \\
                          & \rdTo_{\mu}        & \dTo_{\nu}  \\
                          &                    & \Msheaf^\bullet|_{Z}
  \end{diagram}
  into a diagram which is commutative up to a homotopy $\gamma \in \Cperf(Z)^{-1}(\Esheaf^\bullet|_Z, \Msheaf^\bullet|_{Z})$, i.e.,
  \begin{displaymath}
    d \gamma + \mu - \nu \phi = 0.
  \end{displaymath}
  In other words, $(\mu, \nu, \gamma)$ forms a closed morphism of degree $0$ from $(\Esheaf^\bullet, \Nsheaf^\bullet, \phi)$ to $R(\Msheaf^\bullet|_{W_{s+1}}) = (\Msheaf^\bullet|_{K_{s+1}}, \Msheaf^\bullet|_{W_s}, \kappa)$ in $\Cperf(K_{s+1}) \times^h_{\Cperf(Z)} \Cperf(W_s)$.

  On the other hand, there is a unique way to extend the composition of quasi-isomorphisms $\sigma \circ \phi: (\Esheaf^\bullet|_Z, d_{\Esheaf}) \to \alpha_{W_s}(E_s)|_{Z}$ of complexes of sheaves of $\Osheaf_{\Yhat|_{Z}}$-modules to a morphism
    \[ \psi: \alpha_{K_{s+1}}(E)|_{Z} \to \alpha_{W_s}(E_s)|_{Z} \]
  of sheaves of $\Asheaf^\bullet_{\Yhat}|_Z$-modules, such that there is a commutative diagram
  \begin{diagram}[width=2.6em]
    \Esheaf^\bullet|_{Z}     & \rTo^{\quad \phi \quad}    & \Nsheaf^\bullet|_{Z}  \\
    \dTo^{\rho}              &                  & \dTo_{\sigma}  \\
    \alpha_{K_{s+1}}(E)|_Z   & \rTo_{\quad \psi \quad}    & \alpha_{W_s}(E_s)|_{Z}.
  \end{diagram}
  We use the same notation for the map between global sections induced by $\psi$:
    \[ \psi: E|_{Z} \to E_s|_{Z} \quad \text{in} \quad \PP_{A(Z)}. \]
  It is a quasi-isomorphism between the total complexes and hence a homotopy equivalence between cohesive modules. Thus
    \[ (E, E_s, \psi) \in \PP_{A(K_{s+1})} \times^h_{\PP_{A(Z)}} \PP_{A(W_s)}. \]
  Moreover, $(\rho,\sigma,0)$ forms a closed morphism of degree zero
    \[ (\rho,\sigma,0): (\Esheaf^\bullet, \Nsheaf^\bullet, \phi) \to \alpha_{K_{s+1}} \times \alpha_{W_s} (E, E_s, \psi) =  (\alpha_{K_{s+1}}(E), \alpha_{W_s}(E_s), \psi) \]
  in $\Cperf(K_{s+1}) \times^h_{\Cperf(Z)} \Cperf(W_s)$.

  By Theorem \ref{thm:descent_formalnbhd}, there exits $E_{s+1} = (E^\bullet_{s+1}, \Econn_s) \in \PP_{A(W_{s+1})}$ such that $R(E_{s+1})$ is homotopy equivalent to $(E, E_s, \psi)$ in $\PP_{A(K_{s+1})} \times^h_{\PP_{A(Z)}} \PP_{A(W_s)}$. Consider the following diagram (Lemma \ref{lemma:comm_diagram_perfect})
  \begin{diagram}[height=2.5em]
    \PP_{A(W_{s+1})}          & \rTo^{R}   & \PP_{A(K_{s+1})} \times^h_{\PP_{A(Z)}} \PP_{A(W_s)}  \\
    \dTo_{\alpha_{W_{s+1}}}   &            & \dTo_{\alpha_{K_{s+1}} \times \alpha_{W_s}}  \\
    \Cperf(W_{s+1})           & \pile{\rTo^{\quad R \quad} \\ \lTo_{A}}  & \Cperf(K_{s+1}) \times^h_{\Cperf(Z)} \Cperf(W_s).
  \end{diagram}
  By Lemma \ref{lemma:A_quasi-isom}, $A \circ (\alpha_{K_{s+1}} \times \alpha_{W_s}) \circ R (E_{s+1})$ and $A \circ (\alpha_{K_{s+1}} \times \alpha_{W_s}) (E, E_s, \psi)$ are quasi-isomorphic. But by the discussion above we also have morphisms in $\Cperf(K_{s+1}) \times^h_{\Cperf(Z)} \Cperf(W_s)$
  \begin{equation}\label{eq:roof}
    R(\Msheaf^\bullet|_{W_{s+1}})
      \xleftarrow{(\mu,\nu,\gamma)} (\Esheaf^\bullet, \Nsheaf^\bullet, \phi)
      \xrightarrow{(\rho,\sigma,0)} \alpha_{K_{s+1}} \times \alpha_{W_s} (E, E_s, \psi)
  \end{equation}
  with $\mu$, $\nu$, $\rho$, $\sigma$ all being quasi-isomorphisms of complexes of sheaves over $K_{s+1}$ or $W_s$. By applying the functor $A$ on both morphisms in \eqref{eq:roof} and again by Lemma \ref{lemma:A_quasi-isom}, we obtain an isomorphism in $D_{pe}(W_{s+1})$
  \begin{equation}
    A \circ R(\Msheaf^\bullet|_{W_{s+1}})
      \simeq A \circ (\alpha_{K_{s+1}} \times \alpha_{W_s}) (E, E_s, \psi)
  \end{equation}
  and hence also
  \begin{equation}\label{eq:finale-1}
    A \circ (\alpha_{K_{s+1}} \times \alpha_{W_s}) \circ R (E) \simeq A \circ R(\Msheaf^\bullet|_{W_{s+1}}).
  \end{equation}
  in $D_{pe}(W_{s+1})$. Now apply Lemma \ref{lemma:A_quasi-isom} once again and Lemma \ref{lemma:adjunction_quasi-isom}, we have isomorphisms in $D_{pe}(W_{s+1})$
  \begin{equation}\label{eq:finale-2}
    A \circ (\alpha_{K_{s+1}} \times \alpha_{W_s}) \circ R (E)
      \simeq A \circ R \circ \alpha_{W_{s+1}} (E)
      \simeq \alpha_{W_{s+1}} (E)
  \end{equation}
  and
  \begin{equation}\label{eq:finale-3}
    \Msheaf^\bullet|_{W_{s+1}} \simeq A \circ R(\Msheaf^\bullet|_{W_{s+1}}).
  \end{equation}
  Finally, by combining \eqref{eq:finale-1}, \eqref{eq:finale-2} and \eqref{eq:finale-2} we conclude that
    \[ \Msheaf^\bullet|_{W_{s+1}} \simeq \alpha_{W_{s+1}}(E) \]
  in $D_{pe}(W_{s+1})$. The assertion for $s+1$ is therefore verified.
\end{proof}


\appendix

\section{Flatness of $\Asheaf^\bullet_{\Yhat}$ over $\Osheaf_{\Yhat}$}\label{appendix:flatness}

\subsection{Preliminaries on completions}\label{appendix:flatness_completions}

\subsubsection{Completions of modules and sheaves}

We first recall basics of adic topology and its completion, with main reference to \cite{AtiyahMacdonald}. Throughout this section, by a ring we mean a unital commutative ring. Let $A$ be any ring and $\Jideal$ an ideal of $A$. The filtration of ideals $A = \Jideal^0 \supseteq \Jideal^1 \supseteq \Jideal^2 \supseteq \cdots$ defines the \emph{$\Jideal$-adic topology on $A$}, and the completion $\hat{A}$ of $A$ is the topological ring $\hat{A} = \varprojlim_{n} A / \Jideal^n$. Suppose $M$ is an topological $A$-module whose topology is determined by a descending filtration $(M_n): M = M_0 \supseteq M_1 \supseteq M_2 \supseteq \cdots$ with $M_n$ being open submodules of $M$. Its completion is $\hat{M} = \varprojlim_{n} M / M_n$. $(M_n)$ is a \emph{$\Jideal$-filtration} if $\Jideal$-filtration if $\Jideal M_n \subseteq M_{n+1}$ for all $n$, and then $\hat{M}$ is naturally an $\hat{A}$-module. A $\Jideal$-filtration $(M_n)$ is a \emph{stable $\Jideal$-filtration} if $\Jideal M_n = M_{n+1}$ for sufficiently large $n$. It is a fact that topology defined by any stable $\Jideal$-filtration of $M$ coincides with the \emph{$\Jideal$-adic topology on $M$} defined by the $\Jideal$-filtration $(\Jideal^n M)$. The following lemma is standard.

\begin{lemma}\label{lemma:completion_finiteness}
  \begin{enumerate}[i)]
    \item \label{lemma:completion_finiteness_i}
      For any ring $A$, if $M$ is a finitely generated $A$-module, then $\hat{A} \otimes_A M \to \hat{M}$ is surjective, where $\hat{M}$ is the completion of $M$ with respect to its $\Jideal$-adic topology. In particular, $\hat{M}$ is a finitely generated $\hat{A}$-module.
    \item \label{lemma:completion_finiteness_ii}
      If $\Jideal$ is a finitely generated ideal of $A$, then $\hat{\Jideal} = \hat{A} \Jideal$ and $(\Jideal^n)\sphat~ = \hat{\Jideal}^n$ for any $n \geq 0$.
  \end{enumerate}
\end{lemma}

\begin{proof}
  For \ref{lemma:completion_finiteness_i}), see Prop 10.13., p. 108, \cite{AtiyahMacdonald}. \ref{lemma:completion_finiteness_ii}) is immediate from \ref{lemma:completion_finiteness_i}) by setting $M = \Jideal$ and $\Jideal^n$.
\end{proof}

From now on, in any situation when there is an $A$-module $M$ with topology determined by submodules $(M_n)$ and a submodule $M'$, by $\hat{M}'$ we always mean the completion of $M'$ with respect to the topology induced from $M$, i.e., that induced by $(M' \cap M_n)$. The following results can be found in \cite{AtiyahMacdonald} and \cite{Matsumura}.

\begin{lemma}\label{lemma:completion_submodule}
  Let M be a topological $A$-module whose topology is determined by open submodules $(M_n)$. Let $M'$ be submodule of $M$.
  \begin{enumerate}[i)]
    \item
      $\hat{M}'$ can be identified with the closure of $\rho(M')$ in $\hat{M}$, where $\rho: M \to \hat{M}$ is the canonical $A$-module homomorphism.
    \item
      If moreover $M'$ is an open submodule, then $\hat{M} / \hat{M}' \simeq M / M'$.
    \item
      The topology of $\hat{M}$ is determined by the submodules $(\hat{M}_n) = (\overline{\rho(M_n)})$.
  \end{enumerate}
\end{lemma}

We can define formal completions of coherent analytic sheaves over $Y$ along the submanifold $X$ in a similar manner. To every $\Osheaf_Y$-sheaf $\Fsheaf$ we associate the $\Osheaf_\Yhat$-sheaf
\begin{displaymath}
  \hat{\Fsheaf} := \varprojlim_{r} \Fsheaf / \Isheaf^r \Fsheaf,
\end{displaymath}
which is called \emph{the restriction of $\Fsheaf$ to the formal neighborhood $\Yhat$}. Note that we have $\hat{\Osheaf} = (\Osheaf_Y)\sphat ~ = \Osheaf_\Yhat$.

Moreover, every $\Osheaf_Y$-morphism $\varphi: \Fsheaf \to \Gsheaf$ determines an $\Osheaf_\Yhat$-morphism $\hat{\varphi}: \hat{\Fsheaf} \to \hat{\Gsheaf}$, thus we have a functor from the category of $\Osheaf_\Yhat$-modules to the category of $\Osheaf_\Yhat$-modules. In order to study properties of this functor, especially exactness, we need a device for computing stalks $\hat{\Fsheaf}_x$ of $\hat{\Fsheaf}$. It is necessary to point out that, in general, $\Osheaf_{\Yhat,x} \neq (\Osheaf_{Y,x})\sphat~$ for $x \in X$, where $(\Osheaf_{Y,x})\sphat~$ is the completions of $\Osheaf_{Y,x}$ with respect to the ideal $\Isheaf_x \subset \Osheaf_{Y,x}$.

We recall basic properties of completions of coherent sheaves, of which the proof can be found in \cite{BanicaStanasila}, Chap. VI.
\begin{theorem}\label{thm:Yhat_elementary}
  \begin{enumerate}[i)]
    \item
      The sheaf $\Osheaf_\Yhat$ is a coherent sheaf of rings.
    \item
      If $\Fsheaf$ is $\Osheaf_Y$-coherent, then $\hat{\Fsheaf}$ is $\Osheaf_\Yhat$-coherent and $\hat{\Fsheaf} \simeq \Fsheaf \otimes_{\Osheaf_Y} \Osheaf_{\Yhat}$.
    \item
      The functor $\Fsheaf \mapsto \hat{\Fsheaf}$ is exact on coherent sheaves.
    \item \label{thm:Yhat_elementary_projlim}
      For every $\Osheaf_\Yhat$-coherent sheaf $\Fsheaf$ we have $\hat{\Fsheaf} = \varprojlim_{r} \Fsheaf / \Isheaf^r \Fsheaf$.
  \end{enumerate}
\end{theorem}

\subsubsection{Completed tensor products}

We recall the notion of complete tensor products of modules with adic topologies from \cite{EGA1}. Let $A$ be any ring and $\Jideal$ an ideal of $A$. Let $M$ and $N$ be two $A$-modules with $\Jideal$-filtrations $(M_n)$ and $(N_n)$ respectively. Then we can endow the algebraic tensor $M \otimes_A N$ with a $\Jideal$-filtration
\begin{equation}\label{eq:tensor_filtration}
  F_n(M \otimes_A N) = M_n \otimes_A N + M \otimes_A N_n, \quad n \geq 0.
\end{equation}
The completion of $M \otimes_A N$ with respect to this filtration is
\begin{equation}
  \begin{split}
    &M \hat{\otimes}_A N
      = (M \otimes_A N)\sphat
      = \varprojlim_n M \otimes_A N / (M_n \otimes_A N + M \otimes_A N_n) = \\
      &= \varprojlim_n (M/M_n) \otimes_{A/\Jideal^n} (N/N_n),
  \end{split}
\end{equation}
is called the \emph{complete tensor product} of $M$ and $N$. Note that if one of the modules is endowed with the $\Jideal$-adic topology, say, $N$, i.e., $N_n = \Jideal^n N$, then the filtration on the tensor becomes
\begin{displaymath}
  F_n(M \otimes_A N) = M_n \otimes_A N
\end{displaymath}
since $M \otimes_A \Jideal^n N = \Jideal^n M \otimes_A N \subseteq M_n \otimes_A N$.

Note that even when $A$, $M$ and $N$ are all complete under their topologies, it is not guaranteed that algebraic tensor $M \otimes_A N$ is complete with respect to the filtration \eqref{eq:tensor_filtration}. Nevertheless we still have canonical isomorphisms
\begin{equation}
  M \hat{\otimes}_A N \simeq \hat{M} \hat{\otimes}_{A} N \simeq \hat{M} \hat{\otimes}_{\hat{A}} \hat{N},
\end{equation}
since $M/M_n$ can be naturally identified with $\hat{M}/\hat{M}_n$ and the same for $N$. From now on, all the completions are with respect to the adic-topologies unless otherwise mentioned.

Similarly we can define complete tensor products of sheaves with filtrations. Suppose $\Fsheaf$, $\Gsheaf$ are two $\Osheaf_{\Yhat}$-modules with descending $\hat{\Isheaf}$-filtrations of $\Osheaf_{\Yhat}$-submodules $(\Fsheaf_n)$ and $(\Gsheaf_n)$ respectively, i.e., $\hat{\Isheaf} \Fsheaf_n \subseteq \Fsheaf_{n+1}$, $\hat{\Isheaf} \Gsheaf_n \subseteq \Gsheaf_{n+1}$ for all $n \geq 0$. Then the \emph{completed tensor product $\Fsheaf \hat{\otimes}_{\hat{\Osheaf}} \Gsheaf$ of $\Fsheaf$ and $\Gsheaf$} is the completion of the usual tensor $\Fsheaf \otimes_{\Osheaf_{\hat{Y}}} \Gsheaf$ with respect to the $\hat{\Isheaf}$-filtration of subsheaves $(\Fsheaf_n \otimes_{\hat{\Osheaf}} \Gsheaf + \Fsheaf \otimes_{\hat{\Osheaf}} \Gsheaf_n)$, i.e.,
\begin{displaymath}
  \Fsheaf \hat{\otimes}_{\hat{\Osheaf}} \Gsheaf := \varprojlim_n \Fsheaf \otimes_{\hat{\Osheaf}} \Gsheaf / (\Fsheaf_n \otimes_{\hat{\Osheaf}} \Gsheaf + \Fsheaf \otimes_{\hat{\Osheaf}} \Gsheaf_n) = \varprojlim_r (\Fsheaf / \Fsheaf_{r+1} ) \otimes_{\Osheaf_{\hat{Y}^{(r)}}} (\Gsheaf / \Gsheaf_{r+1}).
\end{displaymath}
If one of sheaves, say $\Gsheaf$, is carrying the $\hat{\Isheaf}$-adic filtration, i.e., $(\Gsheaf_n) = (\hat{\Isheaf}^n\Gsheaf)$, then as in the case of rings and modules we have
\begin{displaymath}
  (\Fsheaf_n \otimes_{\hat{\Osheaf}} \Gsheaf + \Fsheaf \otimes_{\hat{\Osheaf}} \Gsheaf_n) = (\Fsheaf_n \otimes_{\hat{\Osheaf}} \Gsheaf)
\end{displaymath}
and hence
\begin{displaymath}
  \Fsheaf \hat{\otimes}_{\hat{\Osheaf}} \Gsheaf = \varprojlim_n \Fsheaf \otimes_{\hat{\Osheaf}} \Gsheaf / (\Fsheaf_n \otimes_{\hat{\Osheaf}} \Gsheaf) = \varprojlim_r (\Fsheaf / \Fsheaf_{r+1}) \otimes_{\Osheaf_{\hat{Y}^{(r)}}} \Gsheaf^{(r)}
\end{displaymath}
where $\Gsheaf^{(r)} = \Gsheaf / \hat{\Isheaf}^{r+1} \Gsheaf$, $r \geq 0$. Note that if both $\Fsheaf$ and $\Gsheaf$ are coherent $\Osheaf_{\Yhat}$-modules with $\hat{\Isheaf}$-adic filtrations, then the algebraic tensor $\Fsheaf \otimes_{\hat{\Osheaf}} \Gsheaf$ is also coherent and thus complete with respect to the $\hat{\Isheaf}$-filtration by Theorem \ref{thm:Yhat_elementary}, \ref{thm:Yhat_elementary_projlim}). Hence we have an isomorphism $\Fsheaf \otimes_{\hat{\Osheaf}} \Gsheaf \simeq \Fsheaf \hat{\otimes}_{\hat{\Osheaf}} \Gsheaf$.

\subsubsection{Two useful lemmas}

We prove a slightly generalized version of Artin-Rees lemma with weakened conditions, which will be used in \S ~ \ref{appendix:flatness_main}.

\begin{lemma}\label{lemma:generalized_ArtinRees}
  Let $A$ be a ring and $\Jideal$ a finitely generated ideal of $A$, such that the $\Jideal$-adic completion $\hat{A}$ of $A$ is a noetherian ring. Let $M$ be a finitely generated $A$-module and $M'$ a submodule of $M$. Then:
  \begin{enumerate}[i)]
    \item
      The two $\hat{\Jideal}$-filtrations $(\hat{\Jideal}^n \hat{M}')$ and $((M' \cap \Jideal^n M)\sphat~)$ define the same topology on $\hat{M}'$, i.e., the $\hat{\Jideal}$-adic topology on $\hat{M}'$ coincides its submodule topology in $\hat{M}$.
    \item
      Moreover, we have $\hat{\Jideal}^n \hat{M}' = (\Jideal^n M')\sphat~$. Thus the completions of $M'$ with respect to the two $\Jideal$-filtrations $(\Jideal^n M')$ and $(M' \cap \Jideal^n M)$ coincide.
    \item \label{lemma:generalized_ArtinRees_tensor}
      More generally, if $N$ is an $A$-module with the $\Jideal$-adic topology, then the completed tensor products of $M'$ and $N$ with respect to the two filtrations $(\Jideal^n M')$ and $(M' \cap \Jideal^n M)$ of $M'$ coincide, i.e.,
      \begin{displaymath}
        \varprojlim_n M' \otimes_A N / (\Jideal^n M' \otimes_A N) \simeq \varprojlim_n M' \otimes_A N / ((M' \cap \Jideal^n M) \otimes_A N).
      \end{displaymath}
  \end{enumerate}
\end{lemma}

\begin{proof}
  $\hat{M}$ is a finitely generated $\hat{A}$-module be Lemma \ref{lemma:completion_finiteness}. Observe that the topology of $\hat{M}$ is determined by the $\hat{\Jideal}$-filtration $((\Jideal^n M)\sphat~) = (\hat{\Jideal}^n \hat{M})$, i.e., it is the $\hat{\Jideal}$-adic topology. Indeed, by Lemma \ref{lemma:completion_finiteness} we have $(\Jideal^n M)\sphat = \hat{A} \cdot \rho(\Jideal^n M) = \hat{\Jideal}^n \hat{M}$.

  Now since $\hat{A}$ is noetherian and $\hat{M}$ is finitely generated, $(\hat{M}' \cap \hat{\Jideal}^n \hat{M})$ is a stable $\hat{\Jideal}$-filtration of $\hat{M}'$ by Artin-Rees lemma. But it also defines the topology of $\hat{M}'$ as a submodule of $\hat{M}$, hence any stable $\hat{\Jideal}$-filtration of $\hat{M}'$, in particular, $(\hat{\Jideal}^n \hat{M}')$ also determines its topology. On the other hand, the topology on $\hat{M}'$ is determined by the filtration $((M' \cap \Jideal^n M)\sphat~)$ by Lemma \ref{lemma:completion_submodule}, iii). Thus i) is proved. Since $\hat{M}$ is finitely generated and $\hat{A}$ is noetherian, the submodules $\hat{\Jideal}^n \hat{M}'$ are also finitely generated and hence complete and closed in $\hat{M}'$. Then from the inclusions $\rho(\Jideal^n M') \subseteq \hat{\Jideal}^n \hat{M}' \subseteq (\Jideal^n M')\sphat~$ we deduce that $\hat{\Jideal}^n \hat{M}' = (\Jideal^n M')\sphat~$ and hence
  \begin{displaymath}
    \hat{M}' = \varprojlim_{n} \hat{M}' / \hat{\Jideal}^n \hat{M}' = \varprojlim_{n} \hat{M}' / (\Jideal^n M')\sphat = \varprojlim_{n} M' / \Jideal^n M'
  \end{displaymath}
  (the last equality is by Lemma \ref{lemma:completion_submodule}, ii)), which proves ii).

  For iii), observe that when $M'$ is endowed with the topology induced from $(M' \cap \Jideal^n M)$, the completed tensor product $M' \hat{\otimes}_A N$ is isomorphic to the completion of $\hat{M}' \otimes_{A} N$ with respect to the topology determined by the filtration $((M' \cap \Jideal^n M)\sphat~ \otimes_{A} N)$, which, by i) and ii), is the same as the topology induced by the filtration $((\Jideal^n M')\sphat~ \otimes_{A} N)$. The latter gives rise to the other completion:
  \begin{displaymath}
    \begin{split}
      &\varprojlim_n \hat{M}' \otimes_{A} N / ((\Jideal^n M')\sphat~ \otimes_{A} N)
      = \varprojlim_n (\hat{M}' / (\Jideal^n M')\sphat~) \otimes_{A} N
      \simeq \varprojlim_n (M' / \Jideal^n M') \otimes_{A} N = \\
      &= \varprojlim_n M' \otimes_A N / (\Jideal^n M' \otimes_A N).
    \end{split}
  \end{displaymath}
\end{proof}

The following lemma can be regarded as a version of Nakayama's lemma without assuming that the module $M$ is finitely generated in advance. It is Exercise 7.4 from Chap. 7, \cite{Eisenbud}. A proof can be obtained by applying Prop. 10.24., Chap. 10, \cite{AtiyahMacdonald}.

\begin{lemma}\label{lemma:generalized_Nakayama}
  Let $A$ be a ring that is complete with respect to an ideal $\Jideal$ and $M$ a module over $A$. Suppose $M$ is separated, i.e., $\cap_k \Jideal^k M = 0$, and that the images of $m_1, \ldots, m_n \in M$ generate $M / \Jideal M$. then $m_1, \ldots, m_n$ generate $M$.
\end{lemma}

\subsection{Some Stein theory}\label{appendix:flatness_Stein}

We will show that coherent analytic sheaves on the formal neighborhood $\Yhat$ have various nice behaviors over Stein compacts. We will also prove analogues of Cartan's Theorem A and B (see Theorem \ref{thm:Theorem_A_Yhat} and Theorem \ref{thm:Theorem_B_Yhat}) for $\Yhat$. The proofs we provide here are similar to the arguments in Chap. VI, \cite{BanicaStanasila}, with appropriate changes. Also Theorem A and B for $\Yhat$ should be already well-known and of no surprise to experts. We reproduce them because it is difficult for us to find them written down in the literature.

\subsubsection{Semi-analytic subsets}

We need the notion of \emph{semi-analytic Stein compacts}. We recall briefly the definitions and basic properties from \cite{semianalytic} and \cite{Frisch}.

\begin{definition}
A \emph{Stein compact (subset)} of a complex analytic manifold (or space) $(X, \Osheaf_X)$ is a compact subset of $X$ which admits a fundamental system of Stein open neighborhoods.
\end{definition}

\begin{definition}
  Let $M$ be a real analytic manifold. A subset $A$ of $M$ is said to be \emph{semi-analytic} if and only if for any point $y \in M$ there exits an open neighborhood $U$ of $y$, such that
  \begin{displaymath}
    A \cap U = \bigcup_{i=1}^{p} \bigcap_{j=1}^{q} A_{ij},
  \end{displaymath}
  where each $A_{ij}$ is of the form $\{ f_{ij}(x)=0 \}$ or $\{ f_{ij}(x) > 0 \}$, with $f_{ij}(x)$ a real analytic functions defined over $U$.
\end{definition}

The class of semi-analytic subsets of $M$ is closed under finite union, finite intersection and complement. In particular, intersection of a semi-analytic subset with a closed submanifold $N$ of $M$ is still semi-analytic in $M$. Moreover, a subset of $N$ is semi-analytic in $N$ iff it is semi-analytic in $M$. A finite product of semi-analytic sets is semi-analytic.

\begin{example}
  Any closed polydisc in $\complex^n$ of the form $\{ (z_1,\ldots,z_n) | |z_i| \leq r_i, r_i > 0 \}$ is a semi-analytic Stein compact. Thus any point in a complex analytic manifold (or space) admits a fundamental system of neighborhoods which are semi-analytic Stein compacts.
\end{example}

The proof of main results in this section relies on the following important result about semi-analytic Stein compacts. See \cite{Frisch} and \cite{Siu}.

\begin{theorem}\label{thm:Stein_compacts}
  Let $K$ be a Stein compact in a complex space $(X, \Osheaf_X)$. Then the ring $\Gamma(K, \Osheaf_X)$ is noetherian if  and only if for every set $A$ analytic in a neighborhood of $K$ the set $A \cap K$ has only finitely many connected components.

  In particular, for every semi-analytic Stein compact $K$ in $X$ the ring $\Gamma(K, \Osheaf_X)$ is noetherian.
\end{theorem}

\subsubsection{Cartan's Theorem A $\&$ B for $\Yhat$}

We first recall the following proposition concerning commutativity between the project limit functor and the cohomology functor from Prop. 1.9., Chap. VI, \S 1, \cite{BanicaStanasila}. Also see Prop. 13.3.1., \cite{EGA3}, for a more general version.

\begin{proposition}[\cite{BanicaStanasila}]\label{prop:projlim_cohomology}
  Let $(\Fsheaf_n)_{n \geq 0}$ be a projective system of sheaves of abelian groups over a topological space $X$. Suppose the following conditions are fulfilled:
  \begin{enumerate}[i)]
    \item
      There is a base $\mathscr{B}$ for the topology of $X$ such that for any $U \in \mathscr{B}$, $H^q(U,\Fsheaf_n) = 0$ for all $q \geq 1$, $n \geq 0$
    \item
      For any $U \in \mathscr{B}$, the morphisms $\Gamma(U, \Fsheaf_{n+1}) \to \Gamma(U, \Fsheaf_n)$, $n \geq 0$, are surjective.
  \end{enumerate}
  Under these assumptions, the canonical morphism
  \begin{displaymath}
    h_m : H^m(X, \varprojlim_n \Fsheaf_n) \to \varprojlim_n H^m(X, \Fsheaf_n)
  \end{displaymath}
  is surjective for any $m \geq 0$. If in addition for some $m$ the projective system $(H^{m-1}(X, \Fsheaf_n))_{n \geq 0}$ satisfies the Mittag-Leffler (ML) condition, then $h_m$ is bijective.
\end{proposition}

\begin{theorem}[Theorem B for $\hat{Y}$]\label{thm:Theorem_B_Yhat}
  Let $U \subseteq Y$ be a Stein open subset. Then for any $\Fsheaf \in \Coh(\hat{Y})$, the cohomology groups $H^q(U, \Fsheaf)$ vanish for $q \geq 1$.
\end{theorem}

\begin{proof}
  By Theorem \ref{thm:Yhat_elementary}, \ref{thm:Yhat_elementary_projlim}), $\Fsheaf = \varprojlim_n \Fsheaf_n$, where $\Fsheaf_n = \Fsheaf^{(n)} = \Fsheaf / \hat{\Isheaf}^{n+1} \Fsheaf$, $n \geq 0$, are coherent $\Osheaf_Y$-modules. All the Stein open subsets of $U$ form a base $\mathscr{B}$ for the topology of $U$. Thus by Cartan's Theorem B for Stein spaces, for the inverse system $(\Fsheaf_n)_{n \geq 0}$ and any $V \in \mathscr{B}$ the conditions i) and ii) of Proposition \ref{prop:projlim_cohomology} are satisfied. Moreover, since $U$ is itself Stein, again by Theorem B we have $H^q(U, \Fsheaf_n) = 0$ for $q \geq 1$ and the connecting maps in the projective system $(H^0(U, \Fsheaf_n))_{n \geq 0}$ are all surjective and hence the system satisfies the ML condition. Thus Proposition \ref{prop:projlim_cohomology} applies and we get
  \begin{displaymath}
    H^q(U, \Fsheaf) = H^q (U, \varprojlim_{n \geq 0} \Fsheaf_n) \simeq \varprojlim_{n \geq 0} H^q(U, \Fsheaf_n) = 0
  \end{displaymath}
  for $q \geq 1$.
\end{proof}

\begin{corollary}\label{cor:Theorem_B_compact}
  Let $K \subseteq Y$ be a Stein compact. Then for any $\Fsheaf \in \Coh(\hat{Y})$, the cohomology groups $H^q(K, \Fsheaf)$ vanish for $q \geq 1$.
\end{corollary}

\begin{proof}
  Since $K$ is compact, we have the canonical isomorphisms
  \begin{displaymath}
    \varinjlim_{U} H^q (U, \Fsheaf) \simeq H^q (K, \Fsheaf), \quad \forall ~ q \geq 0,
  \end{displaymath}
  where $U$ runs over all Stein open neighborhoods of $K$ in $Y$ (see Prop. 2.5.1 and Remark 2.6.9, \cite{Kashiwara}).
\end{proof}

\begin{theorem}[Theorem A for $\Yhat$]\label{thm:Theorem_A_Yhat}
  Let $\Fsheaf \in \Coh(\Yhat)$ and $U$ be any Stein open subset of $Y$ (or X). Then the sections of $\Fsheaf$ over $U$ generate the stalk $\Fsheaf_x$ as $\hat{\Osheaf}_x$-module for each $x \in U$.
\end{theorem}

\begin{proof}
  The theorem can be deduced from Theorem \ref{thm:Theorem_B_Yhat} in the same way as the ordinary Theorem A being deduced from the ordinary Theorem B (see, e.g., \S ~ 3.2., Chap. III, \cite{SeveralComplex}).
\end{proof}

\begin{corollary}\label{cor:Theorem_A_compact}
  Let $\Fsheaf \in \Coh(\hat{Y})$ and $K$ be a Stein compact in $Y$. Then there exists an exact sequence of the form $\hat{\Osheaf}^p \to \hat{\Osheaf}^q \to \Fsheaf \to 0$ over (some Stein open neighborhood of) $K$. In particular, $\Gamma(K, \Fsheaf)$ is a finitely generated $\Gamma(K, \hat{\Osheaf})$-module.
\end{corollary}

\begin{remark}
  We can now take away the condition in Lemma \ref{lemma:Yhat_Stein_tensor}, i), that $\Fsheaf$ admits a free resolution over $K$ and the result is still true. Yet we want to point out that, for sake of proving Theorem \ref{thm:exactness_tensor_Dolbeault},Theorem A for $\Yhat$ is not necessary since we can always choose $K$ to be small enough so that free resolutions of $\Fsheaf$ exist over $K$.
\end{remark}

We now recall Lemma 1.2., Chap. VI, \cite{BanicaStanasila}, and then prove its analogue for $\Osheaf_{\hat{Y}}$-modules. The statement and proof for the case of $\hat{Y}$ are almost the same as that in \cite{BanicaStanasila} with only minor adjustments.

\begin{lemma}[\cite{BanicaStanasila}]\label{lemma:Y_Stein_tensor}
  Let $K$ be a Stein compact in $Y$. Then:
  \begin{enumerate}[i)]
    \item
      For any $\Fsheaf, \Gsheaf \in \Coh(Y)$, the canonical morphism
      \begin{displaymath}
        \Gamma(K,\Fsheaf) \otimes_{\Gamma(K,\Osheaf)} \Gamma(K,\Gsheaf) \to \Gamma(K,\Fsheaf \otimes_{\Osheaf} \Gsheaf).
      \end{displaymath}
      is an isomorphism.
    \item \label{lemma:Yhat_Stein_tensor_ideal}
      For any $\Fsheaf \in \Coh(Y)$ and whenever $\Jsheaf$ is coherent ideal sheaf of $\Osheaf_Y$,
      \begin{displaymath}
        \Gamma(K, \Jsheaf\Fsheaf) = \Gamma(K, \Jsheaf) \cdot \Gamma(K,\Fsheaf).
      \end{displaymath}
  \end{enumerate}
\end{lemma}

\begin{lemma}\label{lemma:Yhat_Stein_tensor}
  Let $K$ be either a Stein compact or a relatively compact Stein open subset in $Y$. Then:
  \begin{enumerate}[i)]
    \item
      For any $\Fsheaf \in \Coh(\hat{Y})$ admitting an exact sequence of the form $\hat{\Osheaf}^p \to \hat{\Osheaf}^q \to \Fsheaf \to 0$ on (some neighborhood of) $K$, and any $\Osheaf_{\hat{Y}}$-module $\Gsheaf$ which is either coherent or fine, the canonical morphism
      \begin{displaymath}
        \Gamma(K,\Fsheaf) \otimes_{\Gamma(K,\hat{\Osheaf})} \Gamma(K,\Gsheaf) \to \Gamma(K,\Fsheaf \otimes_{\hat{\Osheaf}} \Gsheaf).
      \end{displaymath}
      is an isomorphism.
    \item \label{lemma:Yhat_Stein_tensor_ideal}
      For any $\Osheaf_{\hat{Y}}$-module $\Fsheaf$ which is either coherent or fine,
      \begin{displaymath}
        \Gamma(K, \hat{\Isheaf} \Fsheaf) = \Gamma(K, \hat{\Isheaf}) \cdot \Gamma(K,\Fsheaf).
      \end{displaymath}
  \end{enumerate}
\end{lemma}

\begin{proof}
  \begin{enumerate}[i)]
    \item
      By our assumptions $\Fsheaf \otimes_{\hat{\Osheaf}} \Gsheaf$ is either coherent or fine, thus $H^1(K,\Fsheaf \otimes_{\hat{\Osheaf}} \Gsheaf) = 0$. We then have the following exact commutative diagram \vspace{2pt}
      \begin{diagram}
        & \Gamma(K,\hat{\Osheaf}^p) \otimes_{\Gamma(K,\hat{\Osheaf})} \Gamma(K,\Gsheaf)  & \longrightarrow
        & \Gamma(K,\hat{\Osheaf}^q) \otimes_{\Gamma(K,\hat{\Osheaf})} \Gamma(K,\Gsheaf)  & \longrightarrow
        & \Gamma(K,\Fsheaf) \otimes_{\Gamma(K,\hat{\Osheaf})} \Gamma(K,\Gsheaf)  & \longrightarrow  & 0    \\
        & \dTo  &   & \dTo  &   & \dTo   \\
        & \Gamma(K,\hat{\Osheaf}^p \otimes_{\hat{\Osheaf}} \Gsheaf)   & \longrightarrow
        & \Gamma(K,\hat{\Osheaf}^q \otimes_{\hat{\Osheaf}} \Gsheaf)   & \longrightarrow
        & \Gamma(K,\Fsheaf \otimes_{\hat{\Osheaf}} \Gsheaf)   & \longrightarrow  & 0    \\
      \end{diagram}
      with the first two vertical arrows being isomorphisms. The claim then follows from the five Lemma.
    \item
      Since $\Isheaf$ is a coherent ideal sheaf of $\Osheaf_Y$, there exists an exact sequence $\Osheaf^p_Y \to \Osheaf^q_Y \to \Isheaf \to 0$ of $\Osheaf_Y$-modules over (some Stein open neighborhood of) $K$ by applying Cartan's Theorem B. Then by passing to completion we have an exact sequence $\hat{\Osheaf}^p \to \hat{\Osheaf}^q \to \hat{\Isheaf} \to 0$. Apply i) and note that $\hat{\Isheaf} \Fsheaf$ is the image of the morphism $\hat{\Isheaf} \otimes_{\hat{\Osheaf}} \Fsheaf \to \Fsheaf$.
  \end{enumerate}
\end{proof}

\begin{lemma}\label{lemma:Stein_compact_stalk}
  Let $\Fsheaf \in \Coh(\hat{Y})$.
  \begin{enumerate}[i)]
    \item
      For any Stein compact $K$, the $\Gamma(K,\hat{\Osheaf})$-module $\varprojlim_{r} \Gamma(K,\Fsheaf^{(r)})$ is canonically isomorphic to the completion $\Gamma(K,\Fsheaf)\sphat~$ of the $\Gamma(K, \hat{\Osheaf})$-module $\Gamma(K, \Fsheaf)$ with respect to the $\Gamma(K, \hat{\Isheaf})$-adic topology.
    \item
      For any point $x \in K$ there is a canonical isomorphism
      \begin{displaymath}
        \Fsheaf_x \simeq \varinjlim_{K} \Gamma(K,\Fsheaf)\sphat~,
      \end{displaymath}
      where $K$ is any Stein semianalytic compact which contains $x$ as an interior point.
  \end{enumerate}
\end{lemma}

\begin{proof}
  \begin{enumerate}[i)]
    \item
      From Corollary \ref{cor:Theorem_B_compact} and Lemma \ref{lemma:Yhat_Stein_tensor}, \ref{lemma:Yhat_Stein_tensor_ideal}), we derive the isomorphisms
      \begin{displaymath}
        \begin{split}
          &\varprojlim_r \Gamma(K, \Fsheaf^{(r)})
            = \varprojlim_r \Gamma(K, \Fsheaf / \hat{\Isheaf}^{r+1} \Fsheaf)
            \simeq \varprojlim_r (\Gamma(K, \Fsheaf) / \Gamma(K, \hat{\Isheaf}^{r+1} \Fsheaf)) \simeq  \\
          &\simeq \varprojlim_r (\Gamma(K, \Fsheaf) / \Gamma(K, \hat{\Isheaf})^{r+1} \Gamma(K, \Fsheaf))
             = \Gamma(K,\Fsheaf)\sphat.
        \end{split}
      \end{displaymath}
    \item
      $\displaystyle \Fsheaf_x = \varinjlim_U \Gamma(U, \varprojlim_r \Fsheaf^{(r)}) = \varinjlim_U \varprojlim_r \Gamma(U, \Fsheaf^{(r)}) = \varinjlim_K \varprojlim_r \Gamma(K, \Fsheaf^{(r)}) \simeq \varinjlim_K \Gamma(K,\Fsheaf)\sphat~$, where $U$ is any neighborhood of $x$ and $K$ is as in the statement.
  \end{enumerate}
\end{proof}

The main reason for working on Stein semianalytic compacts instead of Stein open subsets is due to the following lemma:

\begin{lemma}\label{lemma:Yhat_noetherian}
  Let $K$ be a Stein compact in $Y$. Then:
  \begin{enumerate}[i)]
    \item
      $\Gamma(K, \hat{\Isheaf})$ is a finitely generated ideal of the ring $\Gamma(K, \Osheaf_{\Yhat})$.
    \item
      If in addition $K$ is semianalytic, then the $\Gamma(K, \hat{\Isheaf})$-adic completion of $\Gamma(K, \Osheaf_{\hat{Y}})$ is a noetherian ring.
  \end{enumerate}
\end{lemma}

\begin{proof}
  \begin{enumerate}[i)]
    \item
      There exists an exact sequence of coherent $\Osheaf_Y$-modules of the form $0 \to \Gsheaf \to \Osheaf^p_Y \to \Isheaf \to 0$ over some Stein open neighborhood of $K$.  By taking completion we get an exact sequence of coherent $\Osheaf_{\Yhat}$-modules $0 \to \hat{\Gsheaf} \to \hat{\Osheaf}^p \to \hat{\Isheaf} \to 0$. Thus by Corollary \ref{cor:Theorem_B_compact} the morphism $\Gamma(K, \hat{\Osheaf}^p) \to \Gamma(\hat{\Isheaf})$ is surjective.
    \item
      By Lemma \ref{lemma:Yhat_Stein_tensor}, (\ref{lemma:Yhat_Stein_tensor_ideal}), we have
      \begin{displaymath}
        \begin{split}
          &\Gamma(K,\hat{\Osheaf})\sphat~
            = \varprojlim_n \Gamma(K, \hat{\Osheaf}) / \Gamma(K, \hat{\Isheaf})^n
            = \varprojlim_n \Gamma(K, \hat{\Osheaf} / \hat{\Isheaf}^n)
            = \varprojlim_n \Gamma(K, \Osheaf_Y / \Isheaf^n) =  \\
            &= \varprojlim_n \Gamma(K, \Osheaf_Y) / \Gamma(K, \Isheaf)^n
            = \Gamma(K, \Osheaf_Y)\sphat~,
        \end{split}
      \end{displaymath}
      where $\Gamma(K, \Osheaf_Y)\sphat~$ is the $\Gamma(K,\Isheaf)$-adic completion of $\Gamma(K, \Osheaf_Y)$. But by Theorem \ref{thm:Stein_compacts}, $\Gamma(K,\Osheaf_Y)$ is a noetherian ring and hence so is $\Gamma(K,\hat{\Osheaf})\sphat~$.
  \end{enumerate}
\end{proof}

\subsection{Proof of the main results}\label{appendix:flatness_main}

We first prove a version of Artin-Rees lemma for coherent $\Osheaf_{\Yhat}$-modules by applying results from previous sections.

\begin{lemma}\label{lemma:Yhat_ArtinRees}
  Let $\Fsheaf \in \Coh(\hat{Y})$ and $\Fsheaf'$ a coherent subsheaf of $\Fsheaf$. If $\Gsheaf$ is a fine $\Osheaf_{\hat{Y}}$-module with the $\hat{\Isheaf}$-adic filtration, then the completed tensor products of $\Fsheaf'$ and $\Gsheaf$ with respect to the two filtrations $(\hat{\Isheaf}^n \Fsheaf')$ and $(\Fsheaf' \cap \hat{\Isheaf}^n \Fsheaf)$ of $\Fsheaf'$ coincide, i.e.,
  \begin{displaymath}
    \varprojlim_n \Fsheaf' \otimes_{\hat{\Osheaf}} \Gsheaf / (\hat{\Isheaf}^n \Fsheaf' \otimes_{\hat{\Osheaf}} \Gsheaf) \simeq \varprojlim_n \Fsheaf' \otimes_{\hat{\Osheaf}} \Gsheaf / ((\Fsheaf' \cap \hat{\Isheaf}^n \Fsheaf) \otimes_{\hat{\Osheaf}} \Gsheaf).
  \end{displaymath}
\end{lemma}

\begin{proof}
  As in the proof of Lemma \ref{lemma:Stein_compact_stalk}, ii), if there is a projective system of sheaves $(\Fsheaf_n)_{n \geq 0}$ over $Y$, then for any point $x \in Y$ there are canonical isomorphisms
  \begin{displaymath}
    (\varprojlim_n \Fsheaf_n)_x = \varinjlim_U \Gamma(U, \varprojlim_n \Fsheaf_n) = \varinjlim_U \varprojlim_n \Gamma(U, \Fsheaf_n) = \varinjlim_K \varprojlim_n \Gamma(K, \Fsheaf_n),
  \end{displaymath}
  where $U$ is any open neighborhood of $x$ and $K$ is any Stein semianalytic compact which contains $x$ as an interior point. Thus we only need to show that
  \begin{displaymath}
    \varprojlim_n \Gamma(K, (\Fsheaf' / (\Fsheaf' \cap \hat{\Isheaf}^n \Fsheaf)) \otimes_{\hat{\Osheaf}} \Gsheaf)
    \simeq \varprojlim_n \Gamma(K, (\Fsheaf' / \hat{\Isheaf}^n \Fsheaf') \otimes_{\hat{\Osheaf}} \Gsheaf)
  \end{displaymath}
  for any such $K$.

  Indeed, we derive from Lemma \ref{lemma:Yhat_Stein_tensor} the isomorphisms
  \begin{displaymath}
    \begin{split}
      &\quad \varprojlim_n \Gamma(K, (\Fsheaf' / (\Fsheaf' \cap \hat{\Isheaf}^n \Fsheaf)) \otimes_{\hat{\Osheaf}} \Gsheaf) \\
      & = \varprojlim_n ( \Gamma(K, \Fsheaf') / (\Gamma(K, \Fsheaf') \cap \Gamma(K, \hat{\Isheaf})^n \Gamma(K, \Fsheaf)) ) \otimes_{\Gamma(K, \hat{\Osheaf})} \Gamma(K, \Gsheaf)  \\
      & = \Gamma(K, \Fsheaf') \hat{\otimes}_{\Gamma(K, \hat{\Osheaf})} \Gamma(K, \Gsheaf),
    \end{split}
  \end{displaymath}
  where the completed tensor product is with respect to the subspace topology of $\Gamma(K,\Fsheaf') \allowbreak \subseteq \Gamma(K,\Fsheaf)$ and the $\Gamma(K,\hat{\Isheaf})$-adic topology of $\Gamma(K,\Gsheaf)$. By Lemma \ref{lemma:Yhat_noetherian} $\Gamma(K, \hat{\Osheaf})\sphat~$ is noetherian. By Lemma \ref{lemma:Yhat_noetherian}, i), $\Gamma(K,\hat{\Isheaf})$ is a finitely generated ideal of the ring $\Gamma(K,\hat{\Osheaf})$. Moreover, $\Gamma(K,\Fsheaf)$ is a finitely generated $\Gamma(K, \hat{\Osheaf})$-module by Corollary \ref{cor:Theorem_A_compact}. So we can apply Lemma \ref{lemma:generalized_ArtinRees}, \ref{lemma:generalized_ArtinRees_tensor}) to deduce that the completed tensor above is isomorphic to the one with $\Gamma(K,\Fsheaf')$ carrying the $\Gamma(K,\hat{\Isheaf})$-adic topology. Hence
  \begin{displaymath}
    \begin{split}
      &\quad \varprojlim_n \Gamma(K, (\Fsheaf' / (\Fsheaf' \cap \hat{\Isheaf}^n \Fsheaf)) \otimes_{\hat{\Osheaf}} \Gsheaf)  \\
      &\simeq \varprojlim_n \Gamma(K, \Fsheaf') \otimes_{\Gamma(K,\hat{\Osheaf})} \Gamma(K, \Gsheaf) / (\Gamma(K, \hat{\Isheaf})^n \Gamma(K, \Fsheaf') \otimes_{\Gamma(K,\hat{\Osheaf})} \Gamma(K, \Gsheaf))  \\
      &= \varprojlim_n \Gamma(K, (\Fsheaf' / \hat{\Isheaf}^n \Fsheaf') \otimes_{\hat{\Osheaf}} \Gsheaf).
    \end{split}
  \end{displaymath}
\end{proof}

\begin{theorem}\label{thm:exactness_tensor_Dolbeault}
  \begin{enumerate}[i)]
    \item
      The functor $\Fsheaf \mapsto \Fsheaf \hat{\otimes}_{\hat{\Osheaf}} \Asheaf_{\Yhat}$ is exact on coherent sheaves of $\Osheaf_{\Yhat}$-modules, where all the sheaves involved are endowed with $\hat{\Isheaf}$-adic filtrations.
    \item
      Moreover, for any $\Fsheaf \in \Coh(\Yhat)$ the canonical morphism $\Fsheaf \otimes_{\hat{\Osheaf}} \Asheaf_{\Yhat} \to \Fsheaf \hat{\otimes}_{\hat{\Osheaf}} \Asheaf_{\Yhat}$ is an isomorphism.
    \item
      $\Fsheaf \mapsto (\Fsheaf \otimes_{\hat{\Osheaf}} \Asheaf^\bullet_{\Yhat}, 1 \otimes \partialbar)$ gives an exact functor from $\Coh(\Yhat)$ to the abelian category of sheaves of dg-modules over $(\Asheaf^\bullet_{\Yhat}, \partialbar)$.
  \end{enumerate}
\end{theorem}

\begin{proof}
  \begin{enumerate}[i)]
    \item
      Let $0 \to \Fsheaf' \to \Fsheaf \to \Fsheaf'' \to 0$ be an exact sequence in $\Coh(\Yhat)$. By Lemma \ref{lemma:Yhat_ArtinRees}, there are isomorphisms
      \begin{displaymath}
        \begin{split}
          \Fsheaf' \hat{\otimes}_{\hat{\Osheaf}} \Asheaf_{\Yhat}
            &\simeq \varprojlim_r \left(\Fsheaf' / (\Fsheaf' \cap \hat{\Isheaf}^{r+1} \Fsheaf)\right) \otimes_{\hat{\Osheaf}} \Asheaf_{\Yhat}  \\
            &= \varprojlim_r \left( \Fsheaf' / (\Fsheaf' \cap \hat{\Isheaf}^{r+1} \Fsheaf) \right) \otimes_{\hat{\Osheaf} / \hat{\Isheaf}^{r+1}} (\Asheaf_{\Yhat} / \hat{\Isheaf}^{r+1} \Asheaf_{\Yhat})  \\
            &= \varprojlim_r \left( \Fsheaf' / (\Fsheaf' \cap \hat{\Isheaf}^{r+1} \Fsheaf) \right) \otimes_{\Osheaf_{\Yhatfinite}} \Asheaf_{\Yhatfinite}.
        \end{split}
      \end{displaymath}
      We have a projective system of exact sequences of sheaves
      \begin{displaymath}
        0 \to \Fsheaf' / (\Fsheaf' \cap \hat{\Isheaf}^{r+1} \Fsheaf) \to \Fsheaf / \hat{\Isheaf}^{r+1} \Fsheaf \to \Fsheaf'' / \hat{\Isheaf}^{r+1} \Fsheaf'' \to 0.
      \end{displaymath}
      Tensor with $\Asheaf_{\Yhatfinite}$ and by the flatness of $\Asheaf_{\Yhatfinite}$ over $\Osheaf_{\Yhatfinite}$ (Proposition \ref{prop:flatness_Yhatfinite}) we again get a projective system of exact sequences of fine sheaves
      \begin{displaymath}
        \begin{split}
          0 \to (\Fsheaf' / (\Fsheaf' \cap \hat{\Isheaf}^{r+1} \Fsheaf)) \otimes_{\Osheaf_{\Yhatfinite}} \Asheaf_{\Yhatfinite} \to & (\Fsheaf / \hat{\Isheaf}^{r+1} \Fsheaf) \otimes_{\Osheaf_{\Yhatfinite}} \Asheaf_{\Yhatfinite} \to \\
          &\to (\Fsheaf'' / \hat{\Isheaf}^{r+1} \Fsheaf'') \otimes_{\Osheaf_{\Yhatfinite}} \Asheaf_{\Yhatfinite} \to 0,
        \end{split}
      \end{displaymath}
      so that the corresponding projective system of exact sequences of sections over any open subset $U \subseteq Y$ satisfies the ML condition. Thus the projective limit
      \begin{displaymath}
        0 \to \Fsheaf'\hat{\otimes}_{\hat{\Osheaf}} \Asheaf_{\Yhat} \to \Fsheaf \hat{\otimes}_{\hat{\Osheaf}} \Asheaf_{\Yhat} \to \Fsheaf'' \hat{\otimes}_{\hat{\Osheaf}} \Asheaf_{\Yhat} \to 0
      \end{displaymath}
      is still an exact sequence.

    \item
      Since $\Fsheaf \in \Coh(\Yhat)$, locally there is always an exact sequence of the form $\hat{\Osheaf}^p \to \hat{\Osheaf}^q \to \Fsheaf \to 0$. The conclusion follows routinely from the exact commutative diagrams
      \begin{diagram}
        & \hat{\Osheaf}^p \otimes_{\hat{\Osheaf}} \Asheaf_{\Yhat}  & \rTo  & \hat{\Osheaf}^q \otimes_{\hat{\Osheaf}} \Asheaf_{\Yhat}  & \rTo  & \Fsheaf \otimes_{\hat{\Osheaf}} \Asheaf_{\Yhat}  & \rTo  & 0  \\
        & \dTo  &   & \dTo  &   & \dTo  \\
        & \hat{\Osheaf}^p \hat{\otimes}_{\hat{\Osheaf}} \Asheaf_{\Yhat}  & \rTo  & \hat{\Osheaf}^q \hat{\otimes}_{\hat{\Osheaf}} \Asheaf_{\Yhat}  & \rTo  & \Fsheaf \hat{\otimes}_{\hat{\Osheaf}} \Asheaf_{\Yhat}  & \rTo  & 0
      \end{diagram}
      and the fact that $\Asheaf_{\Yhat}$ is itself $\hat{\Isheaf}$-adic complete.
    \item
      This follows immediately from i), ii) and the fact that $\Asheaf^\bullet_{\Yhat}$ is locally free over $\Asheaf_{\Yhat}$.
  \end{enumerate}
\end{proof}

\begin{proposition}\label{prop:Dolbeault_resolution_Yhat_completed}
  Suppose $\Fsheaf$ is a coherent $\Osheaf_{\Yhat}$-module, then its completed tensor with $(\Asheaf^\bullet_{\Yhat}, \partialbar)$ gives an exact sequence of $\Osheaf_{\Yhat}$-modules,
  \begin{displaymath}
    0 \to \Fsheaf \to \Fsheaf \hat{\otimes}_{\hat{\Osheaf}} \Asheaf^0_{\Yhat} \xrightarrow{1 \otimes \partialbar} \Fsheaf \hat{\otimes}_{\hat{\Osheaf}} \Asheaf^1_{\Yhat} \to \cdots \xrightarrow{1 \otimes \partialbar} \Fsheaf \hat{\otimes}_{\hat{\Osheaf}} \Asheaf^m_{\Yhat} \to 0,
  \end{displaymath}
  where $m = \dim X$.
\end{proposition}

\begin{proof}
  By Proposition \ref{prop:Dolbeault_resolution_Yhatfinite}, for any $r \geq 0$ we have exact sequences of sheaves
  \begin{displaymath}
    0 \to \Fsheaf^{(r)} \to \Fsheaf^{(r)} \otimes_{\Osheaf_\Yhatfinite} \Asheaf^0_\Yhatfinite \to \Fsheaf^{(r)} \otimes_{\Osheaf_\Yhatfinite} \Asheaf^1_\Yhatfinite \to \cdots \to \Fsheaf^{(r)} \otimes_{\Osheaf_\Yhatfinite} \Asheaf^m_\Yhatfinite \to 0.
  \end{displaymath}
  Thus over any Stein open subset $U \subseteq Y$, we have a projective system of exact sequences of abelian groups
  \begin{displaymath}
    0 \to \Gamma(U, \Fsheaf^{(r)}) \to \Gamma(U,\Fsheaf^{(r)} \otimes_{\Osheaf_\Yhatfinite} \Asheaf^0_\Yhatfinite) \to \cdots \to \Gamma(U,\Fsheaf^{(r)} \otimes_{\Osheaf_\Yhatfinite} \Asheaf^m_\Yhatfinite) \to 0,
  \end{displaymath}
  where all the connecting morphisms are surjective due to the coherence of $\Fsheaf^{(r)}$ and the fineness of the rest. Hence the ML condition is satisfied and by passing to the projective limit we get an exact sequence
  \begin{displaymath}
    0 \to \Gamma(U, \Fsheaf) \to \Gamma(U, \Fsheaf \hat{\otimes}_{\hat{\Osheaf}} \Asheaf^0_{\Yhat}) \to \cdots \to \Gamma(U, \Fsheaf \hat{\otimes}_{\hat{\Osheaf}} \Asheaf^m_{\Yhat}) \to 0,
  \end{displaymath}
  for each Stein open subset $U$. The claimed result then follows.
\end{proof}

Combining Proposition \ref{prop:Dolbeault_resolution_Yhat_completed} with Theorem \ref{thm:exactness_tensor_Dolbeault}, ii), immediately yields:

\begin{theorem}\label{thm:Dolbeault_resolution_Yhat}
  Suppose $\Fsheaf$ is a coherent $\Osheaf_{\Yhat}$-module, then its (algebraic) tensor with $(\Asheaf^\bullet_{\Yhat}, \partialbar)$ gives an exact sequence of $\Osheaf_{\Yhat}$-modules,
  \begin{displaymath}
    0 \to \Fsheaf \to \Fsheaf \otimes_{\hat{\Osheaf}} \Asheaf^0_{\Yhat} \xrightarrow{1 \otimes \partialbar} \Fsheaf \otimes_{\hat{\Osheaf}} \Asheaf^1_{\Yhat} \to \cdots \xrightarrow{1 \otimes \partialbar} \Fsheaf \otimes_{\hat{\Osheaf}} \Asheaf^m_{\Yhat} \to 0,
  \end{displaymath}
  where $m = \dim X$.
\end{theorem}

Finally, Theorem \ref{thm:exactness_tensor_Dolbeault} and Theorem \ref{thm:Dolbeault_resolution_Yhat} together imply Theorem \ref{thm:Dolbeault_main}.


\bibliographystyle{halpha}
\bibliography{Dolbeault_dga_Bib}

\begin{thebibliography}{GPR94}

\bibitem[AM69]{AtiyahMacdonald}
M.~F. Atiyah and I.~G. Macdonald.
\newblock {\em Introduction to commutative algebra}.
\newblock Addison-Wesley Publishing Co., Reading, Mass.-London-Don Mills, Ont.,
  1969.

\bibitem[BB]{Ben-Bassat-Block}
O.~{Ben-Bassat} and J.~{Block}.
\newblock {Cohesive DG Categories I: Milnor Descent}.
\newblock arXiv:math.AG/\allowbreak 1201.6118v1.

\bibitem[BK91]{BondalKapranov}
A.~I. Bondal and M.~M. Kapranov.
\newblock Enhanced triangulated categories.
\newblock {\em Math. USSR-Sb.}, 70(1):93--107, 1991.

\bibitem[Blo10]{Block1}
J.~Block.
\newblock Duality and equivalence of module categories in noncommutative
  geometry.
\newblock In {\em A celebration of the mathematical legacy of {R}aoul {B}ott},
  volume~50 of {\em CRM Proc. Lecture Notes}, pages 311--339. Amer. Math. Soc.,
  Providence, RI, 2010.

\bibitem[Bor95]{Borel}
E.~Borel.
\newblock Sur quelques points de la th\'eorie des fonctions.
\newblock {\em Ann. Sci. \'Ecole Norm. Sup. (3)}, 12:9--55, 1895.

\bibitem[BS76]{BanicaStanasila}
C.~B{\u{a}}nic{\u{a}} and O.~St{\u{a}}n{\u{a}}{\c{s}}il{\u{a}}.
\newblock {\em Algebraic methods in the global theory of complex spaces}.
\newblock Editura Academiei, Bucharest, 1976.
\newblock Translated from the Romanian.

\bibitem[DK90]{Donaldson}
S.~K. Donaldson and P.~B. Kronheimer.
\newblock {\em The geometry of four-manifolds}.
\newblock Oxford Mathematical Monographs. The Clarendon Press Oxford University
  Press, New York, 1990.
\newblock Oxford Science Publications.

\bibitem[Eis95]{Eisenbud}
D.~Eisenbud.
\newblock {\em Commutative algebra}, volume 150 of {\em Graduate Texts in
  Mathematics}.
\newblock Springer-Verlag, New York, 1995.
\newblock With a view toward algebraic geometry.

\bibitem[FM07]{FiorenzaManetti}
D.~Fiorenza and M.~Manetti.
\newblock {$L_\infty$} structures on mapping cones.
\newblock {\em Algebra Number Theory}, 1(3):301--330, 2007.

\bibitem[Fri67]{Frisch}
J.~Frisch.
\newblock Points de platitude d'un morphisme d'espaces analytiques complexes.
\newblock {\em Invent. Math.}, 4:118--138, 1967.

\bibitem[GH78]{GriffithHarris}
P.~Griffiths and J.~Harris.
\newblock {\em Principles of algebraic geometry}.
\newblock Wiley-Interscience [John Wiley \& Sons], New York, 1978.
\newblock Pure and Applied Mathematics.

\bibitem[GPR94]{SeveralComplex}
H.~Grauert, Th. Peternell, and R.~Remmert, editors.
\newblock {\em Several complex variables. {VII}}, volume~74 of {\em
  Encyclopaedia of Mathematical Sciences}.
\newblock Springer-Verlag, Berlin, 1994.
\newblock Sheaf-theoretical methods in complex analysis, A reprint of {{\i}t
  Current problems in mathematics. Fundamental directions. Vol. 74} (Russian),
  Vseross. Inst. Nauchn. i Tekhn. Inform. (VINITI), Moscow.

\bibitem[Gro60]{EGA1}
A.~Grothendieck.
\newblock \'{E}l\'ements de g\'eom\'etrie alg\'ebrique. {I}. {L}e langage des
  sch\'emas.
\newblock {\em Inst. Hautes \'Etudes Sci. Publ. Math.}, (4):228, 1960.

\bibitem[Gro61]{EGA3}
A.~Grothendieck.
\newblock \'{E}l\'ements de g\'eom\'etrie alg\'ebrique. {III}. \'{E}tude
  cohomologique des faisceaux coh\'erents. {I}.
\newblock {\em Inst. Hautes \'Etudes Sci. Publ. Math.}, (11):167, 1961.

\bibitem[Kap99]{Kapranov}
M.~Kapranov.
\newblock Rozansky-{W}itten invariants via {A}tiyah classes.
\newblock {\em Compositio Math.}, 115(1):71--113, 1999.

\bibitem[Kel94]{Keller1}
B.~Keller.
\newblock Deriving {DG} categories.
\newblock {\em Ann. Sci. \'Ecole Norm. Sup. (4)}, 27(1):63--102, 1994.

\bibitem[Kel06]{Keller2}
B.~Keller.
\newblock On differential graded categories.
\newblock In {\em International {C}ongress of {M}athematicians. {V}ol. {II}},
  pages 151--190. Eur. Math. Soc., Z\"urich, 2006.

\bibitem[KM58]{KoszulMalgrange}
J.-L. Koszul and B.~Malgrange.
\newblock Sur certaines structures fibr\'ees complexes.
\newblock {\em Arch. Math. (Basel)}, 9:102--109, 1958.

\bibitem[KS90]{Kashiwara}
M.~Kashiwara and P.~Schapira.
\newblock {\em Sheaves on manifolds}, volume 292 of {\em Grundlehren der
  Mathematischen Wissenschaften [Fundamental Principles of Mathematical
  Sciences]}.
\newblock Springer-Verlag, Berlin, 1990.
\newblock With a chapter in French by Christian Houzel.

\bibitem[Loj64]{semianalytic}
S.~Lojasiewicz.
\newblock Triangulation of semi-analytic sets.
\newblock {\em Ann. Scuola Norm. Sup. Pisa (3)}, 18:449--474, 1964.

\bibitem[Mal67]{Malgrange}
B.~Malgrange.
\newblock {\em Ideals of differentiable functions}.
\newblock Tata Institute of Fundamental Research Studies in Mathematics, No. 3.
  Tata Institute of Fundamental Research, Bombay, 1967.

\bibitem[Man07]{Manetti}
M.~Manetti.
\newblock Lie description of higher obstructions to deforming submanifolds.
\newblock {\em Ann. Sc. Norm. Super. Pisa Cl. Sci. (5)}, 6(4):631--659, 2007.

\bibitem[Mat80]{Matsumura}
H.~Matsumura.
\newblock {\em Commutative algebra}, volume~56 of {\em Mathematics Lecture Note
  Series}.
\newblock Benjamin/Cummings Publishing Co., Inc., Reading, Mass., second
  edition, 1980.

\bibitem[Ran08]{Ran}
Ziv Ran.
\newblock Lie atoms and their deformations.
\newblock {\em Geom. Funct. Anal.}, 18(1):184--221, 2008.

\bibitem[Siu69]{Siu}
Y.-T. Siu.
\newblock Noetherianness of rings of holomorphic functions on {S}tein compact
  subsets.
\newblock {\em Proc. Amer. Math. Soc.}, 21:483--489, 1969.

\bibitem[Yua]{LinfAlgebroid1}
S.~Yu.
\newblock The dolbeault dga of the formal neighborhood of a diagonal.
\newblock arXiv:1211.1567 [math.AG].

\bibitem[Yub]{LinfAlgebroid2}
S.~Yu.
\newblock {$L_\infty$}-algebroid of a formal neighborhood.
\newblock {\em in preparation}.

\end{thebibliography}

\end{document}